\documentclass[12pt]{article}
 
\usepackage[english]{babel}
\usepackage{amsmath,amsthm}
\usepackage{amsfonts}
\usepackage{graphicx}
\usepackage{epsfig}
\usepackage{epstopdf}
\usepackage{url}
\usepackage{amssymb}
\usepackage{color}
\usepackage[%
breaklinks%
,colorlinks%
,linkcolor=red% internal document links
,anchorcolor=blue%
,pagecolor=blue%
,citecolor=blue%
,bookmarks=false%
]{hyperref}
\theoremstyle{Definition}

\numberwithin{equation}{section}

\newtheorem{proposition}{Proposition}
\newtheorem{theorem}{Theorem}
\newtheorem{corollary}{Corollary}
\newtheorem{lemma}{Lemma}
\newtheorem{remark}{Remark}
\newtheorem{definition}{Definition}
\newtheorem{example}{Example}

\DeclareRobustCommand{\secsym}{\ifmmode\mathsection\else\textsection\fi}

\def\S{{\mathcal{S}}}

\def\Res{{\mathrm{Res}}}
\def\para{\vspace{4 mm}}

\def\para{\vspace{2.5mm}}
\def\cP{{\mathcal P}}
\def\cQ{{\mathcal Q}}

\def\gcd{{\rm gcd}}

\def\S{{\mathcal S}}
\def\cS{{\mathcal S}}
\def\cR{{\mathcal R}}
\def\Resultant{{\rm Res}}

\def\Content{{\rm Content}}
\def\Primpart{{\rm Primpart}}
\def\mult{{\rm mult}}

\def\deg{{\rm deg}}

\def\content{{\rm Content}}

\usepackage{mathrsfs}
\def\proj2{\mathbb{P}^2(\mathbb{K})}
\def\projtres{\mathbb{P}^3(\mathbb{K})}

\def\myS{\mathscr{S}}

\def\myB{\mathscr{B}}
\def\myC{\mathscr{C}}

\def\myL{\mathscr{L}}

\def\myG{\mathscr{G}(\projtres)}
\def\myGG{\myG^\star}
\def\myGS{\mathscr{G}(\proj2)}

\def\cJ{\mathcal{J}}
\def\LC{\mathrm{LC}}

\def\cL{\mathcal{L}}
\def\cA{\mathcal{A}}

\def\cM{\mathcal{M}}

\def\K{\mathbb{K}}
\def\cT{\mathcal{T}}
\def\mapdeg{\mathrm{degMap}}
\def\betap{\mult(\myB(\cP))}
\def\betas{\mult(\myB(\cS))}
\def\betaq{\mult(\myB(\cQ))}

\def\oz{{\overline z}}
\def\ot{{\overline t}}

\def\ox{{\overline x}}

\def\oy{{\overline y}}
\def\myX{\overline{\mathrm{X}}}
\def\myY{\overline{\mathrm{Y}}}

\begin{document}

\title{On the base point locus of surface parametrizations: formulas and consequences}
%\thanks{This work has been partially funded by Ministerio de Econom\'{\i}a y Competitividad  under the Project MTM2017-88796-P. The author S. P\'erez-D\'{\i}az is member of the Research Group ASYNACS (Ref. CCEE2011/R34)}

\author{David A. Cox\\Department of Mathematics \& Statistics\\ Amherst College\\ Amherst, MA 01002 USA\\
dacox@amherst.edu\\
Sonia P\'erez-D\'{\i}az and J. Rafael Sendra\\
Dpto.\ de F\'{\i}sica y Matem\'aticas \\
        Universidad de Alcal\'a \\
      E-28871 Madrid, Spain  \\
sonia.perez@uah.es, rafael.sendra@uah.es}
\date{August 18, 2020}          % Enter your date or \today between curly braces
\maketitle

\begin{abstract}
This paper shows that the multiplicity of the base points locus of a projective rational surface parametrization can be expressed as the degree of the content of a univariate resultant.  As a consequence, we get a new proof of the degree formula relating the degree of the surface, the degree of the parametrization, the base points multiplicity, and the degree of the rational map induced by the parametrization. In addition, we extend both formulas to the case of dominant rational maps of the projective plane and describe how the base point loci of a parametrization and its reparametrizations are related. As an application of these results, we explore how the degree of a surface reparametrization is affected by the presence of base points.
\end{abstract}


\para

\section{Introduction}\label{S-intro}
Let $\mathrm{X}, \mathrm{Y}$ be irreducible projective varieties of the same dimension, and consider a dominant rational map $\Phi=(\Phi_1:\cdots:\Phi_m): \mathrm{X}\dashrightarrow \mathrm{Y}$, where the $\Phi_i$ are homogeneous polynomials of the same degree and $\gcd(\Phi_1,\ldots,\Phi_m)=1$. The base points of $\Phi$ are the elements in the subvariety of $\mathrm{X}$ where $\Phi$ is not defined; that is, the projective variety defined by $\{\Phi_1,\ldots,\Phi_m\}$.  In our case, since we are mainly interested in projective rational parametrizations, $\mathrm{X}$ is the whole projective space; i.e., $\mathrm{X}=\mathbb{P}^{n}$. If $n=1$, then $\mathrm{Y}$ is a curve and $\Phi$ does not have base points. For the surface case, i.e., $n=2$, the base point subvariety is either empty or zero-dimensional. If $n>2$, the dimension of base point locus can be positive.

\para

Base points play an important role in the analysis of unirational varieties, since the explanation of many degenerate behaviours are often based on them. Some examples are, for instance, the  study of the degree of a rational surface by means of the degree of the polynomials in its rational parametrizations (see e.g.\ \cite{Cox2003,Chen,SchichoIssac,Schicho2006,Wang2004,Sed2003}), or the surjective cover of a surface by means of the images of finitely many rational parametrizations (see e.g.\ \cite{CSV,SSV14,SSV}). As a consequence, many authors have studied base points (see  e.g.\ \cite{AdkinsHoffmanWang2005, Chen2005,PS13,Sed1990,Wang2004}).

\para

In this paper, we deal with the problem of \textit{properly counting} the number of base points of  projective rational surface parametrizations. %(see Theorem \ref{theorem-formulaP}).
This question has been treated by many  authors. In \cite{Cox2003}, the problem is addressed for birational  triangular parametrizations, and in
 \cite{Sed2003} the case of  tensor product surfaces is established; see also, \cite{Chen} and \cite{Wang2004}. In addition, \cite{SchichoIssac} introduces the notion of blow up of the base locus, and referring to \cite{Conforto} presents a formula for the case of a birational parametrization. To our knowledge, the first general answer to the problem, in the sense of requiring no additional hypothesis such as the birationality of the parametrization or any particular   structure of the parametrization, appears in \cite{Cox2001}, where the degree formula is proved using Segre classes from Fulton's book \cite{Fulton}. Another proof that applies B\'ezout’s Theorem to two generic linear combinations of the polynomials in the parametrization and uses reduction ideals to relate this to the Hilbert-Samuel multiplicity of the base points appears in an unpublished lecture of the first author \cite{CoxWebPage}.

 \para

In this paper, also without assuming additional hypotheses, we present a formula that relates the multiplicity of the base point locus with the content of a univariate resultant (see Theorems \ref{theorem-betaW1LW2L} and \ref{theorem-BetaP-K1K2}). Furthermore, as a consequence of this relationship, we present an elementary proof of the degree formula (see Theorem \ref{theorem-formulaP}) relating the degree of the surface, the degree of the parametrization, the multiplicity of the base locus and the cardinality of the generic fiber of the parametrization. The proof in this paper was found independently of any previous work. Our methods are based on the intersection theory of curves in combination with well-known results from elimination theory, especially the properties of resultants (see \cite{SWP} and \cite{Walker}).

\para

The usual definition of the multiplicity of the base point locus uses Hilbert-Samuel multiplicities, which can be challenging to compute individually.  In Corollary \ref{cor-HSW1W2K1K2}, we provide a simple computational method to determine the sum of these multiplicities.

\para

Finally, we also state similar formulas for the case of rational maps between the projective plane (see Theorems  \ref{theorem-betaV1LV2L}, \ref{theorem-BetaP-J1J2}   and \ref{theorem-formulapbS}). Moreover, as a consequence, we study the variation of the base locus under reparametrizations (see Theorems \ref{T-propertyP1}, \ref{T-propertyP2}) as well as the behavior of the parametrization degree under reparametrizations (see Theorems \ref{T-propertyP1} and \ref{T-propertyP0} as well as Corollary \ref{cor-noBPQ}).

 \para

For this purpose, in Section \ref{sec-BP-parametrizations} we associate to the given parametrization two plane projective curves, defined over the algebraic closure of a transcendental field extension of the ground field (see \eqref{eq-W} and \eqref{eq-WL}).  Our definition of multiplicity is tailored to our needs and gives a good notion of the multiplicity of the base locus.  In Proposition \ref{prop-HSmult} and Corollary \ref{cor-HSmult}, we show that this agrees with the usual definition via Hilbert-Samuel multiplicity.  Our definition enables us to express the multiplicity of the base locus in terms of  the content (w.r.t.\ the introduced transcendental elements) of the resultant of the two polynomials defining the curves (see Theorem \ref{theorem-betaW1LW2L}). In a second step, we show that the curves can be simplified by introducing fewer transcendentals in the field extension (see \eqref{eq-K1K2g}), so that for almost all projective transformations the content of the resultant of these two new curves also yields the multiplicity of the base locus (see Theorem \ref{theorem-BetaP-K1K2}).  From here, we carefully analyze the primitive part of the resultant of the new curves and relate it to the degree of the surface and the cardinality of the generic fiber of the parametrization (see Lemma \ref{lemma-formula}).  Then the degree formula stated in Theorem \ref{theorem-formulaP} follows immediately.

  \para

In Section \ref{sec-formula-birational}, we show how the results in Section \ref{sec-BP-parametrizations} can be adapted to dominant rational maps from $\mathbb{P}^2$ onto $\mathbb{P}^2$ (see Theorems \ref{theorem-betaV1LV2L}, \ref{theorem-BetaP-J1J2} and \ref{theorem-formulapbS}). Finally, in Section \ref{sec-composition}, we study the behaviour of the base loci of two parametrizations of the same surface, when one is the reparametrization of the other;  see Theorems \ref{T-propertyP1}, \ref{T-propertyP0} and \ref{T-propertyP2} as well as Corollaries \ref{cor-betapqs} and \ref{cor-noBPQ}. In addition, we apply the results developed in Section \ref{sec-composition} to study how the degree of a parametrization varies under the presence of base points. More precisely, let $\cP,\cQ$ be curve  parametrizations related by $\cP=\cQ \circ \cS$, where $\cS$ is a non-constant rational function. Then $\deg(\cP) = \deg(\cQ)\, \deg(\cS)$. However,  for surface parametrizations, this equality is not true in general. In this paper a characterization of the equality is given when $\cP$, $\cQ$ are surface parameterizations and $\cS$ is a dominant rational map of $\mathbb{P}^2$. We show how this characterization is directly related with the base points of $\cQ$ and $\cS$ (see Theorem \ref{T-propertyP1}). Furthermore, we prove that the degree of the composition decreases, i.e., $\deg(\cP) <\deg(\cQ)\,\deg(\cS)$, if and only if
 a certain polynomial gcd is non-trivial, a fact that can be geometrically interpreted by asking
 a base point of $\cQ$ to be in the image of a curve via the rational map $\cS$.  We conclude that if $\cQ$ has no base points, then $\deg(\cP) = \deg(\cQ)\,\deg(\cS)$ (see Corollary \ref{cor-noBPQ}).

\para

The paper concludes with an appendix that explains how the proof of the degree formula given in Theorem~\ref{theorem-formulaP} relates to the unpublished argument sketched in \cite{CoxWebPage}.

\para

\noindent {\bf Notation.} Throughout this paper, we will use the following notation:
\begin{itemize}
\item $\K$ is an algebraically closed field of characteristic zero.
\item For a rational map
\[ \begin{array}{cccc}
\cM: & \mathbb{P}^{k_1}(\K) & \dashrightarrow & \mathbb{P}^{k_2}(\K) \\
     &     \ot=(t_1:\cdots:t_{k_{1}+1}) & \longmapsto & (M_1(\ot):\cdots: M_{k_2+1}(\ot)),
\end{array}
\]
where the non-zero $M_i$ are homogenous polynomial in $\ot$ of the same degree, we denote by $\deg(\cM)$ the degree $\deg_{\ot}(M_i)$, for $M_i$ non-zero, and by $\mapdeg(\cM)$ the degree of the map $\cM$; that is,  the cardinality of the generic fiber of $\cM$ (see e.g.\ \cite{Harris:algebraic}).
\item Let $f\in \mathbb{L}[t_1,t_2,t_3]$ be homogeneous and non-zero, where $\mathbb{L}$ is a field extension of $\K$. Then $\myC(f)$ denotes the projective plane curve defined by $f$ over the algebraic closure of $\mathbb{L}$.
\item $\mathscr{G}(\mathbb{P}^{k}(\K))$ denotes the set of all projective transformations of $\mathbb{P}^{k}(\K)$, and $\mathscr{G}(\mathbb{P}^{k}(\K))^*$ denotes the set of those transformations in $\mathscr{G}(\mathbb{P}^{k}(\K))$ whose matrix representation is of the form
\[ \left(\begin{array}{c|l} A & \mathbf{0}^T \\
\hline \mathbf{0} & 1 \end{array} \right),\]
where $\mathbf{0}=(0,\ldots,0)$.
\end{itemize}

\section{Formula for rational surface parametrizations}\label{sec-BP-parametrizations}

In this section, we consider a projective rational surface  $\myS\subset \projtres$ and a
rational parametrization of $\myS$, namely,
\begin{equation}\label{eq-Param}
\begin{array}{llll}
\mathcal{P}: & \mathbb{P}^{2}(\K) &\dashrightarrow &\myS \subset \mathbb{P}^{3}(\K) \\
& \ot & \longmapsto & (p_1(\ot):\cdots:p_4(\ot)),
\end{array}
\end{equation}
where $\ot=(t_1,t_2,t_3)$ and  the $p_i$ are homogenous polynomials of the same degree such that $\gcd(p_1,\ldots,p_4)=1$. We \textbf{assume} that $p_4$ is not zero.

\para

We will deal with the multiplicity of intersection of curves by means of resultants. For this purpose, in the sequel, we will  \textbf{assume} that  $(0:0:1)\not\in \myC(p_i)$ for all $i\in \{1,\ldots,4\}$.

 \para

The two hypotheses imposed above are technicalities for our reasoning.  We  will see in Remark \ref{rem-hipotesis} that the final formula in Theorem \ref{theorem-formulaP} is also true when they do not hold.

\para

\begin{definition}\label{def-base-point-P}
 A \textsf{base point of $\mathcal{P}$} is an element  $A\in \mathbb{P}^{2}(\K)$ such that $\mathcal{P}(A) = \mathbf{0}$.  We denote by $\myB({\cP})$ the set of base points of ${\cP}$.
 \end{definition}

 \para

We observe that $\myB({\cP})$ consists of the intersection points of the  projective plane curves $\myC(p_i)$.  That is,
 \[ \myB({\cP})=\bigcap_{i=1}^{4} \myC(p_i). \]
Note that $\myB(\cP)$ is either empty or finite since $\gcd(p_1,\ldots,p_4)=1$.

\para

We introduce the  following auxiliary polynomials:
\begin{equation}\label{eq-W}
\begin{array}{lll}
W_1(\ox,\ot):=\sum_{i=1}^{4} x_i\, p_i(t_1,t_2,t_3)
 \\ \noalign{\smallskip}
 W_2(\oy,\ot):=\sum_{i=1}^{4} y_i\, p_i(t_1,t_2,t_3),
 \end{array}
 \end{equation}
 where $x_i, y_i$ are new variables. We will work with the projective  plane curves $\myC(W_{i})$ in $\mathbb{P}^{2}(\mathbb{F})$, where $\mathbb{F}$ is algebraic closure of
$\mathbb{K}(\ox,\oy).$ In this situation, we have the following notion.

\para

\begin{definition}\label{def-multiplicity-BP}
We define the \textsf{multiplicity of a
base point} $A\in \myB(\cP)$ as $\mult_{A}(\myC(W_1),\myC(W_2))$, that is, as the multiplicity of intersection at $A$ of $\myC(W_1)$ and $\myC(W_2)$.

In addition, we define the \textsf{multiplicity of the base points locus of ${\cP}$},  denoted $\mult(\myB(\cP))$, as
\begin{equation}\label{eq-multBP}
\mult(\myB(\cP)):=\sum_{A\in \myB({\cP})} \mult_A(\myC(W_{1}),\myC(W_{2})).
\end{equation}
\end{definition}

\para

 A base point $A$ also has the \textsf{Hilbert-Samuel multiplicity} $e(I_A,R_A)$ (see \cite[4.6]{BH}), where
\begin{equation}\label{eq-HSsetup}
I_A := \langle \tilde{p}_1, \tilde{p}_2, \tilde{p}_3, \tilde{p}_4\rangle \subset   R_A={\cal O}_{{\Bbb P}^2(\mathbb{K}),A}
\end{equation}
and $\tilde{p}_i$ is a local equation of $p_i$ near $A$.  This agrees with the multiplicity defined in Definition \ref{def-multiplicity-BP}, as we now show.

\para

\begin{proposition}\label{prop-HSmult}
For $A \in \myB(\cP)$, we have $e(I_A,R_A) = \mult_{A}(\myC(W_1),\myC(W_2))$.
\end{proposition}

\para

\begin{proof}
Recall that we have the field extension $\mathbb{K} \subset \mathbb{F}$.  Since $W_1$ and $W_2$ are defined over the larger field $\Bbb{F}$, Definition \ref{def-multiplicity-BP} implies that
\[
\mult_{A}(\myC(W_1),\myC(W_2)) = \dim_\mathbb{F} {\cal O}_{{\Bbb P}^2(\mathbb{F}),A}/\langle \widetilde{W}_1,\widetilde{W}_2\rangle,
\]
where $\widetilde{W}_i$ is a local equation of $W_i$ near $A$.  In contrast,
$e(I_A,R_A)$ is defined over the base field $\mathbb{K}$.  Since ${\cal O}_{{\Bbb P}^2(\mathbb{K}),A}$ has dimension 2, the Hilbert-Samuel multiplicity satisfies
\[
\dim_\mathbb{K} R_{A}/I^{d+1}_A = \tfrac{1}{2} e(I_A,R_A) d^2 + \text{terms of lower degree in $d$}
\]
for $d \gg 0$ by the proof of Proposition 4.6.2(b) in \cite{BH}.  Over the larger field $\mathbb{F}$, we also have the Hilbert-Samuel multiplicity $e(I_{A,\mathbb{F}},R_{A,\mathbb{F}})$, where
\[
I_{A,\mathbb{F}} := \langle \tilde{p}_1, \tilde{p}_2, \tilde{p}_3, \tilde{p}_4\rangle \subset   R_{A,\mathbb{F}}:={\cal O}_{{\Bbb P}^2(\mathbb{F}),A}.
\]
Let us show that these Hilbert-Samuel multiplicities are equal. The key point is that for $A \in {\Bbb P}^2(\mathbb{K}) \subset {\Bbb P}^2(\mathbb{F})$, $I_{A,\mathbb{F}}$ and $R_{A,\mathbb{F}}$ are obtained from \eqref{eq-HSsetup} by tensoring with $\Bbb{F}$.  It follows easily that
\[
\dim_\mathbb{F} R_{A,\mathbb{F}}/I^{d+1}_{A,\Bbb{F}} = \dim_\mathbb{K} R_{A}/I^{d+1}_A,
\]
from which we conclude that $e(I_{A,\mathbb{F}},R_{A,\mathbb{F}}) = e(I_A,R_A)$.  Hence the proposition will follow once we prove
\[
e(I_{A,\mathbb{F}},R_{A,\mathbb{F}}) = \dim_\mathbb{F}
R_{A,\Bbb{F}}/\langle \widetilde{W}_1,\widetilde{W}_2\rangle.
\]

By Theorem A.1 of \cite{BCD}, we know that if $S_1$ and $S_2$ are generic linear combinations of $p_1,\dots,p_4$ over $\Bbb{F}$, then
\[
e(I_{A,\mathbb{F}},R_{A,\mathbb{F}}) = e(\langle \widetilde{S}_1,\widetilde{S}_2\rangle, R_{A,\Bbb{F}}).
\]
The proof uses the theory of reduction ideals developed in \cite[4.6]{BH}.  The  field $\Bbb{F}$ contains $x_1,\dots,x_4,y_1,\dots,y_4 \in \Bbb{F}$ that are algebraically independent over $\Bbb{K}$.  These give generic linear combinations $W_1$ and $W_2$, so that
\[
e(I_{A,\mathbb{F}},R_{A,\mathbb{F}}) = e(\langle \widetilde{W}_1,\widetilde{W}_2\rangle, R_{A,\Bbb{F}}).
\]
Since $\widetilde{W}_1,\widetilde{W}_2$ form a regular sequence (this follows from the proof of Lemma \ref{Prop-BasePointS} below), we can use to the well-known fact that for a complete intersection, the Hilbert-Samuel multiplicity is easy to compute:
\[
e(\langle \widetilde{W}_1,\widetilde{W}_2\rangle, R_{A,\Bbb{F}}) = \dim_\mathbb{F}
R_{A,\Bbb{F}}/\langle \widetilde{W}_1,\widetilde{W}_2\rangle.
\]
The proposition follows.
\end{proof}

\para

\begin{corollary}\label{cor-HSmult}
$\mult(\myB(\cP)) =\sum_{A\in \myB({\cP})} e(I_A,R_A)$.
\end{corollary}

\para

In \cite[p.\ 189]{BH}, Bruns and Herzog note that computing Hilbert-Samuel multiplicities ``may be a painful and often impossible task.''
In the sequel, we will see how the sum of Hilbert-Samuel multiplicities in Corollary \ref{cor-HSmult}, when reinterpreted via \eqref{eq-multBP}, can be computed by means of a simple resultant.

For this purpose, for $L\in \myG$, we introduce the polynomials
\begin{equation}\label{eq-WL}
\begin{array}{lll}
 W_{1}^{L}(\ox,\ot) & := &\sum_{i=1}^{4} x_i\, L_i({\cP}(\ot))\in \mathbb{K}(\ox,\oy)[\ot]
 \\ \noalign{\smallskip}
 W_{2}^{L}(\oy,\ot) &:=&\sum_{i=1}^{4} y_i \, L_i({\cP}(\ot))\in
 \mathbb{K}(\ox,\oy)[\ot].
\end{array}
\end{equation}
Note that $W_{i}=W_{i}^{\mathrm{Id}}$, where $\mathrm{Id}$ is the identity map.  In the next proposition, we study some properties of these polynomials in relation with the base points.

\para

\begin{proposition}\label{Prop-BasePointS}
If $L\in \myG$, then:
\begin{enumerate}
\item $\myC(W_{1}^{L}), \myC(W_{2}^{L})$ have no common components.
\item If $P,Q\in \myC(W_{1}^{L})\cap \myC(W_{2}^{L})$ are colinear with $(0:0:1)$ and $P\in \proj2$, then $Q\in \proj2$.
\item $\myB({\cP})=\myC(W_{1}^{L}) \cap \myC(W_{2}^{L})\cap \proj2$.
\item If $A\in \myB({\cP})$, then $$\mult(A,\myC(W_{1}^{L}))=\mult(A,\myC(W_{2}^{L}))=\min\{\mult(A,\myC(p_i))\,|\, i=1,\ldots,4\}.$$
\item If $A\in \myB({\cP})$, then the tangents to $\myC(W_{1}^{L})$ at $A$ {\rm(}{\hskip-1pt}similarly to $\myC(W_{2}^{L})${\rm)}, with the corresponding multiplicities, are the factors in $\mathbb{K}[\ox,\ot]\setminus \mathbb{K}[\ox]$ of
    \[ \epsilon_1 x_1 T_1+\epsilon_2 x_2 T_2 +\epsilon_3 x_3 T_3+\epsilon_4 x_4 T_4,\]
    where $T_i$ is the product of the tangents, counted with multiplicities, of $\myC(L_i({\cP}))$ at $A$, and where $\epsilon_i=1$ if $\mult(A,\myC(L_i({\cP}))))=\min\{\mult(A,\myC(L_i({\cP})))\,|\, i=1,\ldots,4\}$ and $0$ otherwise.
\end{enumerate}
\end{proposition}

\begin{proof} Without loss of generality, we may assume that $L$ is indeed the identity map, and hence it is enough to prove the result for $W_1,W_2$.

\noindent (1) If the two curves share a component, then $1\neq B:=\gcd(W_1,W_2)\in \K[\ot]$. Then
$B$ divides $\gcd(p_1,\ldots,p_4)=1$, a contradiction.

\medskip

\noindent (2) Let $\mathbb{F}$ be the algebraic closure of $\K(\ox,\oy)$. Suppose that $Q\in \mathbb{P}^{2}(\mathbb{F})\setminus \proj2$. The line $\myL$ passing through $P=(\lambda:\mu:\rho)$ and $(0:0:1)$ is  $\lambda t_2=\mu t_1$, with $\lambda,\mu\in \K$.
We assume w.l.o.g.\ that $\mu\neq 0$ and hence $\myL$ is of the form $t_1=\gamma t_2$ for some $\gamma\in \K$. If $Q$ is at infinity, i.e., $Q=(a:b:0)$, then since $a=\gamma b$, we have $Q=(\gamma:1:0)\in \proj2$. So we can assume that $Q$ is affine. Consider the polynomials $A_i(t_2,t_3):=W_{i}(\gamma t_2: t_2: t_3)$. Since $Q\in \myC(W_1)\cap \myL$, $Q$ can be expressed as
\[ Q=(\gamma \alpha :\alpha:1) \]
where $\alpha$ is a root of $A_1$; note that $\alpha$ is in the algebraic closure of $\K(\ox)$. Similarly, since $Q\in \myC(W_2)\cap \myL$, then $Q$ is also expressible as
\[
Q=(\gamma \beta :\beta :1)
\]
where $\beta$ is a root of $A_2$; note that $\beta$ is in the algebraic closure of $\K(\oy)$. Therefore, $\alpha=\beta$ is a root of $\gcd(W_1,W_2)\in \K[t_2,t_3]$. So $Q\in \proj2$.

\medskip

\noindent (3) Let $A\in \myB({\cP})$. Then, clearly $A\in \proj2$. Moreover,
$p_i(A)=0,\,\,i=1,\ldots,4$. So, $W_{1}(A)=0=W_2(A)$. Therefore, $A\in \myC(W_{1}) \cap \myC(W_{2})\cap \proj2$. Conversely, if $A\in \myC(W_{1}) \cap \myC(W_{2})\cap \proj2$, then, since $A\in \proj2$, we get $p_i(A)=0$ for all $i$, and hence $A\in \myB({\cP})$.

\medskip

\noindent (4) Changing coordinates, we may assume that $A=(0:0:1)$. Then $p_i$ can be expressed as
\[
p_i(\ot)=M_{i,n_{{\cP}}}(t_1,t_2)+\cdots +M_{i,n_{\cP}-\ell_i}(t_1,t_2)t_3^{\ell_i},
\]
where $M_{i,k}$ is homogeneous of degree $k$, $\ell_i:=\mult(A,\myC(p_i))$,  and $n_{\cP}:=\deg({\cP})=\deg(p_i)$. Moreover,
$W_{1}^{L}$ can be expressed as (similarly for $W_{2}^{L}$)
\[
W_{1}^{L}=N_{n_{\cal P}}(\ox,\oy,t_1,t_2)+\cdots +N_{n_{\cal P}-\ell}(\ox,\oy,t_1,t_2)\,t_3^{\ell},
 \]
where $N_{k}$ is $\{t_1,t_2\}$-homogeneous of degree $k$ and $\ell=\min \{\ell_1,\ell_2,\ell_3\}$. Indeed,
if we define $M_{i,j}=0$ if $j<n_{\cal P}-\ell_i$, then
\begin{equation}\label{eq-tg}
N_{n_{\cal P}-\ell}(\ox,\oy,t_1,t_2)=x_1 M_{1,n_{\cal P}-\ell}+ x_2 M_{2,n_{\cal P}-\ell} +x_3 M_{3,n_{\cal P}-\ell}+x_4 M_{4,n_{\cal P}-\ell}.
\end{equation}
From here, the result follows.

\medskip

\noindent (5) For $L$ being the identity and $A=(0:0:1)$, the result follows from \eqref{eq-tg}.
Now, the general case follows by taking into account how tangents change via a projective transformation and using the fact that multiplicities are preserved.
\end{proof}

Taking into account Proposition \ref{Prop-BasePointS} (1), (2), and the relation between resultants and the multiplicity of intersections (see Chapter IV, Section 5 in \cite{Walker}), we get the next theorem  which relates the multiplicity of the base locus with resultants.

\para

\begin{theorem}\label{theorem-betaW1LW2L} For every $L\in \myG$,
we have
$$\betap=\deg_{\ot}(\content_{\{\ox,\oy\}}(\Res_{t_3}(W_{1}^{L},W_{2}^{L}))). $$
\end{theorem}
\begin{proof}
By hypothesis, $(0:0:1)\not\in \myC(W_{i}^{L})$ for $i=1,2$.  By Proposition \ref{Prop-BasePointS} (2), any intersection point in $\Bbb{P}^2(\Bbb{F})$ colinear with $(0:0:1)$ and a base point lies in $\Bbb{P}^2(\Bbb{K})$ and hence is a base point by Proposition \ref{Prop-BasePointS} (3). Since the curves do not share components by Proposition \ref{Prop-BasePointS} (1), the result follows from Theorem 5.3 of \cite[p.\ 111]{Walker}.
\end{proof}

\para

In the second part of the section, we will show that for almost all projective transformations, the multiplicity of intersection of the base points locus can be achieved by a simplified version of the curves $\myC(W_{i}^{L})$. More precisely,  consider the polynomials
\begin{equation}\label{eq-K1K2g}
\begin{array}{lll}
K_{1}^{L}(\ox,\ot) &:=& W_{1}^{L}(x_4,0,0,-x_1,\ot)=x_4L_1({\cP})-x_1L_4({\cP}) \in \mathbb{K}(\ox)[\ot]\\
\noalign{\medskip}
K_{2}^{L}(\ox,\ot) &:=& W_{2}^{L}(0,0,x_4,-x_3,\ot)=x_4L_3({\cP})-x_3L_4({\cP}) \in \mathbb{K}(\ox)[\ot],
\end{array}
\end{equation}
where $L=(L_1:\cdots:L_4)\in \myG$. We start with a technical lemma that
relates $\Res_{t_3}(K_{1}^{L},K_{2}^{L})\neq 0$ to the resultant $\Res_{t_3}(W_{1}^{L},W_{2}^{L})$ when $L$ lies in a suitably chosen open subset of $\myGG \subsetneq \myG$ (see the notation in Section \ref{S-intro}).

\para

\begin{lemma}\label{lemm-open-1}
Let $\cA=(x_4,0,0,-x_1,0,0,x_4,-x_3,t_1,t_2)$.  Then there exists a non-empty Zariski open subset $\Omega$ of $\myGG$ such that
for every $L\in \Omega$, we have
\begin{enumerate}
\item  $\Res_{t_3}(W_{1}^{L},W_{2}^{L})(\cA)=
\Res_{t_3}(K_{1}^{L},K_{2}^{L}) \ne 0$.
\item $\Primpart_{\ox}(\Res_{t_3}(W_{1}^{L},W_{2}^{L})(\cA))=$
    $\Primpart_{\ox}(\Res_{t_3}(K_{1}^{L} ,K_{2}^{L}))$.
\item $\Content_{\ox}(\Res_{t_3}(W_{1}^{L},W_{2}^{L})(\cA))=$
    $\Content_{\ox}(\Res_{t_3}(K_{1}^{L} ,K_{2}^{L}))$.
\item $\deg_{\ot}(\Primpart_{\{\ox,\oy\}}(\Res_{t_3}(W_{1}^{L},W_{2}^{L}))) = \deg_{\ot}(
\Primpart_{\ox}(\Res_{t_3}(K_{1}^{L},K_{2}^{L})))$.
\item $\deg_{\ot}(\Content_{\{\ox,\oy\}}(\Res_{t_3}(W_{1}^{L},W_{2}^{L})))=\deg_{\ot}(\Content_{\ox}(\Res_{t_3}(K_{1}^{L},K_{2}^{L})))$.
%is primitive w.r.t. $\ox$,
\end{enumerate}
\end{lemma}

\begin{proof}
Let $\cL(u_1,\ldots,u_4)=(\cL_1:\cL_2:\cL_3:u_4)$ be a generic element of $\myGG$; that is, $\cL_i=z_{i,1} u_1+ z_{i,2} u_2+ z_{i,3} u_3$, where $z_{i,j}$ are undetermined coefficients satisfying that the determinant of the corresponding matrix is not zero.
 Let $\oz=(z_{1,1},\ldots,z_{3,3})$. We introduce some notation:
 \begin{itemize}
 \item $W_{1}^{\cL}:=\sum_{i=1}^{4} x_i \cL_i(\cP)$, $W_{2}^{\cL}:=\sum_{i=1}^{4} y_i \cL_i(\cP)$, see \eqref{eq-WL}.
 \item $K_{1}^{\cL}:=x_4 \cL_1(\cP)-x_1\cL_4(\cP)$, $K_{2}^{\cL}:=x_4 \cL_3(\cP)-x_3\cL_4(\cP)$, see \eqref{eq-K1K2g}.
\item $\myX:=(\mathrm{X}_1,\ldots,\mathrm{X}_4)$ are new variables; similarly for $\myY$.  Let $\widetilde{W}_{1}:=W_{1}(\myX,\ot),$ see \eqref{eq-W}; similarly for $\widetilde{W}_{2}$.
    %=\mathrm{X}_1 p_1(\ot)+\cdots +\mathrm{X}_4 p_4(\ot)$,
\item $R^{\cL}(\oz,\ox,\oy,t_1,t_2):=\Res_{t_3}(W_{1}^{\cL},W_{2}^{\cL})$, $S^{\cL}(\oz,\ox,t_1,t_2):=\Res_{t_3}(K_{1}^{\cL},K_{2}^{\cL})$.
    \item     $T(\myX,\myY,t_1,t_2):=\Res_{t_3}(\widetilde{W}_{1},\widetilde{W}_2)$.
 \item   $\cA_1:=(x_4,0,0,-x_1,t_1,t_2)$, and  $\cA_2:=(0,0,x_4,-x_3,t_1,t_2)$.
 \item $\cT:=(x_4z_{1,1}, x_4z_{1,2}, x_4z_{1,3}, -x_1,x_4z_{3,1}, x_4z_{3,2}, x_4z_{3,3}, -x_3,t_1,t_2)$.
 \item For   $L\in \myGG$, we denote by $\oz^L$ the coefficient list of $L$. In addition, we denote by $\cT^L$ the tuple $\cT$ specialized at the coefficients of $L$.
 \end{itemize}

 We now prove statement (1). First observe that for  $L\in \myGG$, the summands $x_4 p_4$ of $W_1^L$ and $x_1 p_4$ of $K_1^L$ do not depend on $L$.  It follows that $\deg_{t_3}(K_{1}^{L}) = \deg_{t_3}(W_{1}^{L})$, and $\deg_{t_3}(K_{2}^{L}) = \deg_{t_3}(W_{2}^{L})$ holds similarly.   Since $K_1^L = W_1^L(\cA_1)$ and  $K_2^L = W_2^L(\cA_2)$, we can apply Lemma 4.3.1 in \cite[p.\ 96]{Winkler} on the specialization of resultants to obtain
\[
\Res_{t_3}(W_{1}^{L},W_{2}^{L})(\cA)= \Res_{t_3}(W_{1}^{L}(\cA_1),W_{2}^{L}(\cA_2)) =
\Res_{t_3}(K_{1}^{L},K_{2}^{L}),
\]
proving the first part of statement (1).  However, to ensure that the resultant is non-zero, we need to put some restrictions on $L \in \myGG$.  We construct a non-empty open subset $\Omega_1 \subset \myGG$ as follows.  Consider
\[
G_1(\oz,t_1,t_2):=\Res_{t_3}(\cL_1(\cP),p_4)\in \K[\oz,t_1,t_2].
\]
Let us show that $G_1\neq 0$. Indeed, if $G_1=0$, then $\gcd(z_{1,1}p_1+z_{1,2} p_2 +z_{1,3} p_3,p_4)\neq 1$.  Since $p_{4}\in \K[\ot]$, this gcd divides $p_1,p_2,p_3,p_4$, which contradicts $\gcd(p_1,\ldots,p_4)=1$. Then define $B_1(\oz)$ to be any non-zero coefficient of $G_1$ w.r.t.\ $\{t_1,t_2\}$.

Similarly, consider
\[
G_2(\oz,t_1,t_2):=\Res_{t_3}(\cL_3(\cP),p_4)\in \K[\oz,t_1,t_2],
\]
and reasoning as above shows that $G_2\neq 0$.  Let $B_2(\oz)$ be any non-zero coefficient of $G_2$ w.r.t.\ $\{t_1,t_2\}$.
Then define $\Omega_{1}$ as
\begin{equation}\label{eq-omega1def}
\Omega_{1}=\{ L\in \myGG \mid B_1(\oz^L)\,B_2(\oz^L)\neq 0\}.
\end{equation}
It follows that $\gcd(L_1(\cP),L_4(\cP))=\gcd(L_3(\cP),L_4(\cP)))=1$ for all $L \in \Omega_{1}$.

Now suppose $\Res_{t_3}(K_{1}^{L},K_{2}^{L}) = 0$ for some $L \in \Omega_1$.  Then $K_{1}^{L} = x_4 L_1(\cP) - x_1L_4(\cP)$ and $K_{2}^{L} = x_4 L_3(\cP) - x_3L_4(\cP)$ have a non-trivial common factor which must divide $L_1(\cP)$, $L_3(\cP)$ and $L_4(\cP)$.  This is impossible since $L \in \Omega_1$, and statement (1) is proved.

Statements (2) and (3) now follow when $L \in \Omega_1$ since $\Res_{t_3}(K_{1}^{L},K_{2}^{L}) \ne 0$.  For statements (4) and (5), our arguments will require that we shrink $\Omega_1$ slightly.  This will lead to the open subset $\Omega$ in the statement of the lemma.

 Before actually constructing $\Omega$, we need some preliminary work that will be useful
 below.
 Let $\LC_{t_3}$ denote the leading coefficient w.r.t.\ $t_3$. Since $p_i(0,0,1)\neq 0$, we know that $\deg_{t_3}(p_1)=\cdots=\deg_{t_3}(p_4) =\deg(\cP)$. Then
 \begin{equation}\label{eq-LCoef}
 \begin{array}{l}
 A_{1}^{*}(\oz,\ox,t_1,t_2):=\LC_{t_3}(W_{1}^{\cL})=\left(\sum_{i=1}^{3} x_{i}\,\sum_{j=1}^{3} z_{i,j}\LC_{t_3}(p_i)\right)+ x_4 \LC_{t_3}(p_4),\\
 \noalign{\medskip}
 A_{2}^{*}(\oz,\oy,t_1,t_2):=\LC_{t_3}(W_{2}^{\cL})
 =\left(\sum_{i=1}^{3} y_{i}\,\sum_{j=1}^{3} z_{i,j}\LC_{t_3}(p_i)\right)+ y_4 \LC_{t_3}(p_4). \end{array}
 \end{equation}
  So  $A_{1}^{*}(\oz,\cA_1)\neq 0$ and $A_{2}^{*}(\oz,\cA_2)\neq 0$. Moreover, we observe that for all $L\in \myGG$, we have $\LC_{t_3}(W_{1}^{L})\neq 0$ since it contains the summand $x_4\LC_{t_3}(p_4)$ that does not depend on $\oz^L$; similarly $\LC_{t_3}(W_{2}^{L})\neq 0$. Then using  the behaviour of the resultant under a ring homomorphism (see \cite[Lemma 4.3.1]{Winkler}), we obtain
\begin{equation}\label{eq-specialization-RL}
R^{\cL}(\oz^L,\ox,\oy,t_1,t_2)=\Res_{t_3}(W_{1}^{L},W_{2}^{L}).
\end{equation}
Analogous reasoning applied to $K_{i}^{\cL}$ yields
 \begin{equation}\label{eq-specialization-SL}
S^{\cL}(\oz^L,\ox,t_1,t_2)=\Res_{t_3}(K_{1}^{L},K_{2}^{L}).
\end{equation}
On the other hand, a  direct algebraic manipulation shows that $\widetilde{W}_i(\cT)=K_{i}^{\cL}$, and similarly as above one gets
\begin{equation}\label{eq-T-S}
T(\cT)=S^{\cL}(\oz,\ox,t_1,t_2).
\end{equation}

 Let us now construct $\Omega$. For this purpose, we introduce the polynomials $A_1, A_2, A_3$ as follows.

 \medskip

 \noindent \textsf{Definition of $A_1$.} By Proposition \ref{Prop-BasePointS} (3), we know that
 $T\neq 0$. Let us show that $R^\cL\neq 0$. Indeed, if it is zero, then $B:=\gcd(W_{1}^{\cL},W_{2}^{\cL})\neq 1$. Thus $B$ divides $x_1 \cL_1(\cP)+x_2 \cL_2(\cP)+x_3 \cL_3(\cP)+x_4 p_4$ and $y_1 \cL_1(\cP)+y_2 \cL_2(\cP)+y_3 \cL_3(\cP)+y_4 p_4$. So,
 $B$ divides $p_4$ and also divides $\cL_i(\cP)$ for $i\in \{1,2,3\}$. In particular $B\in \K[\ot]$ and $B$ divides $\sum_{i=1}^{3} z_{1,i} p_i$. That is, $B$ also divides $p_1,p_2,p_3$. Hence $B$ divides $\gcd(p_1\ldots,p_4)=1$, a contradiction.

Now factor $T$
 as product of the content and the primitive part w.r.t.\  $\{\myX,\myY\}$, and  $R^{\cL}$ as product of the content and the primitive part w.r.t.\ $\{\ox,\oy\}$.  This gives
$T(\myX,\myY,t_1,t_2)= C^*(t_1,t_2) M^*(\myX,\myY,t_1,t_2)$ and $R^{\cL}(\oz,\ox,\oy,t_1,t_2) = C(\oz,t_1,t_2) M(\oz,\ox,\oy,t_1,t_2)$.
Taking $L$ as the identity in  Theorem \ref{theorem-betaW1LW2L}, and using Proposition \ref{Prop-BasePointS} (3), we see that $C^*$ is the factor generated by the base points with the corresponding multiplicities of intersection. Moreover, the same argument applies to $C$ for $L$ generic in $\myGG$, namely $\cL$. Therefore, if $B(t_1,t_2)$ is the factor coming from the based points, then $C^*=B$ and $C=N B$ for some $N\in \K[\oz,t_1,t_2]$. Let us show that $N \in \K[\oz]$.
Indeed, by  Theorem \ref{theorem-betaW1LW2L},  $\deg_{\ot}(B)=\mult(\myB(\cP)).$
Now suppose that $N$ depends on $\{t_1,t_2\}$. Then taking $L$ such that  $N(\oz^{L},t_1,t_2)$ is non-constant, by \eqref{eq-specialization-RL},
$
\deg_{\ot}(\Content_{\{\ox,\oy\}}(\Res_{t_3}(W_{1}^{L},W_2^{L})))>\deg_{\ot}(B)=\mult(\myB(\cP))$,
which contradicts Theorem \ref{theorem-betaW1LW2L}.
So we have
\begin{equation}\label{eq-resTRL}
\begin{array}{rll} T(\myX,\myY,t_1,t_2)&= &B(t_1,t_2) M^*(\myX,\myY,t_1,t_2)
\\
R^{\cL}(\oz,\ox,\oy,t_1,t_2) &= & N(\oz) B(t_1,t_2) M(\oz,\ox,\oy,t_1,t_2).
\end{array}
\end{equation}
 We define the polynomial $A_1$ as follows using $M^*(\cT)$. Observe that since by definition $M^*(\myX,\myY,t_1,t_2)$ is primitive w.r.t.\ $\{\myX,\myY\}$,   $M^*(\cT)$ is primitive w.r.t.\ $\{\ox,\oz\}$. Therefore the resultant
 \[
 E(\ox,\oz,t_1)=\Resultant_{t_2}(M^*(\cT), B(t_1,t_2))
 \]
 is non-zero. Since $E$ is homogeneous w.r.t.\ $t_1$, $E$ is of the form $E=D(\ox,\oz) t_{1}^{m}$ for some $m\in \mathbb{N}$, with $D\neq 0$. Let $e(\oz)$ be  a non-zero coefficient of $D$ w.r.t.\ $\ox$.
 In this situation, we define $A_1(\oz)=N(\oz)\,e(\oz)$.

 \medskip

 \noindent \textsf{Definition of $A_2$.}  Let $M$ be as in \eqref{eq-resTRL}.
Let us show that  $M(\oz,\cA)\neq 0$. Indeed, if $M(\oz,\cA)=0$, then $R^{\cL}(\oz,\cA)=0.$ Using the behaviour of the resultant under a ring homomorphism (see \cite[Lemma 4.3.1]{Winkler}), we have
\[0=R^{\cL}(\oz,\cA)=
\LC_{t_3}(W_{1}^{\cL})(\oz,\cA_1)^{\beta}\, \Res_{t_3}(W_{1}^{\cL}(\oz,\cA_1),
W_{2}^{\cL}(\oz,\cA_2)).
 \]
for $\beta=|\deg_{t_3}(W_{1}^{\cL}(\oz,\cA_1))-\deg_{t_3}(W_{2}^{\cL}(\oz,\cA_2))|$.
As noted above, $\LC_{t_3}(W_{1}^{\cL})(\oz,\cA_1)\neq 0$, and hence $\Res_{t_3}(W_{1}^{\cL}(\oz,\cA_1),W_{2}^{\cL}(\oz,\cA_2))=0$.
 Thus $\gcd(W_{1}^{\cL}(\oz,\cA_1),W_{2}^{\cL}(\oz,\cA_2))\neq 1$, i.e.,
 $x_4 \cL_1 ({\cP})-x_1 \cL_4({\cP})$ and $x_4\cL_3({\cP})-x_3\cL_4({\cP})$ have a common factor. Reasoning as above, this factor divides $\gcd(p_1,\ldots,p_4)=1$, a contradiction.

Let $Q(\oz,\ox)$ be a non-zero coefficient of $M(\oz,\cA)$ w.r.t.\ $\{t_1,t_2\}$.  We define the polynomial $A_2(\oz)$ to be any non-zero coefficient of $Q$ w.r.t.\ $\ox$.

\medskip

 \noindent \textsf{Definition of $A_3$.} Consider the resultant (see \eqref{eq-resTRL})
 \[
 G(\ox,\oz,t_1)=\Resultant_{t_2}(M(\oz,\ox,\oy,t_1,t_2), B(t_1,t_2)).
 \]
 $G\neq 0$ because $M$ is primitive w.r.t.\ $\{\ox,\oy\}$.   Since $G$ is homogeneous w.r.t.\ $t_1$, we have $G=D^*(\oz,\ox,\oy) t_{1}^{m}$ for some $m\in \mathbb{N}$, and some $D^*\in \K[\oz,\ox,\oy]\setminus \{0\}$. Let $g(\oz)$ be  a non-zero coefficient of $D^*$ w.r.t.\ $\{\ox,\oy\}$.
 In this situation, we define $A_3(\oz)=g(\oz)$.

\medskip

We define $\Omega$ to consist of those projective transformations $L\in \Omega_1$ from \eqref{eq-omega1def} such that $A_1(\oz^L) \cdot A_2(\oz^L) \cdot A_{3}(\oz^L)\neq 0 $. Let us prove that  statements (4) and (5) of the lemma hold for $L\in \Omega$.  We begin with the following equalities:
\[
\begin{array}{rclr}
N(\oz^L) B(t_1,t_2) M(\oz^L,\cA) &=& R^{\cL}(\oz^L,\cA) & \text{see \eqref{eq-resTRL}} \\
& = & \Res_{t_3}(W_{1}^{L},W_{2}^{L})(\cA) & \text{see \eqref{eq-specialization-RL}}  \\
& =& \Res_{t_3}(K_{1}^{L},K_{2}^{L}) & \text{see statement (1)}  \\
&=& S^{\cL}(\oz^{L},\ox,t_1,t_3) & \text{see \eqref{eq-specialization-SL}}  \\
& =& T(\cT^L) & \text{see \eqref{eq-T-S}}  \\
& = & B(t_1,t_2) M^*(\cT^L). & \text{see \eqref{eq-resTRL}}
\end{array}
\]
Therefore, since $A_1(\oz^L)\neq 0$, we have $N(\oz^L)\neq 0$ and hence $M(\oz^L,\cA)=M^*(\cT^L)$ up to multiplication by a non-zero field element.  Furthermore, since $e(\oz^L)\neq 0$, we see that $M^*(\cT^L)$ is primitive w.r.t.\ $\ox$, and thus $M(\oz^L,\cA)$ also.  In this situation,
using \eqref{eq-specialization-RL} and \eqref{eq-resTRL} we obtain
\[
\Res_{t_3}(W_{1}^{L},W_{2}^{L})(\cA)= N(\oz^L) B(t_1,t_2) M(\oz^L, \cA).
\]
Moreover, since $M(\oz^L,\cA)$ is primitive w.r.t.\ $\ox$ we get
\begin{equation}\label{eq-primpart1}
\Primpart_{\ox}(\Res_{t_3}(W_{1}^{L},W_{2}^{L})(\cA))=M(\oz^L,\cA).
\end{equation}
On the other hand, applying \eqref{eq-specialization-SL}, \eqref{eq-T-S} and \eqref{eq-resTRL}, we have
\[ B(t_1,t_2) M^*(\cT^L)=T(\cT^L)=\Res_{t_3}(K_{1}^{L},K_{2}^{L}),
\]
and since $M^*(\cT^L)$ is primitive w.r.t.\ $\ox$, we get
\begin{equation}\label{eq-primpart2}
\Primpart_{\ox}(\Res_{t_3}(K_{1}^{L},K_{2}^{L}))=M^*(\cT^L).
\end{equation}
By statement (2), we have $M(\oz^L,\cA)=M^*(\cT^L)$, so
\begin{equation}\label{eq-deg1}
\deg_{\ot}(\Primpart_{\ox}(\Res_{t_3}(K_{1}^{L},K_{2}^{L})))=\deg_{\ot}(M(\oz^L,\cA)).
\end{equation}
Furthermore, since $A_2(\oz^L)\neq 0$, we have $\deg_{\ot}(M(\oz^L,\cA))=\deg_{\ot}(M(\oz,\cA))$. On the other hand, we have seen above that $M(\oz,\cA)\neq 0$. So, since $M(\oz,\ox,\oy,t_1,t_2)$ is homogeneous w.r.t.\ $\{t_1,t_2\}$, we have $\deg_{\ot}(M(\oz,\cA))=\deg_{\ot}(M(\oz,\ox,\oy,t_1,t_2))$.
 Finally, since $A_3(\oz^L)\neq 0$, we get that $M(\oz^L,\ox,\oy,t_1,t_2)$ is primitive w.r.t.\ $\{\ox,\oy\}$, and hence, by \eqref{eq-specialization-RL} and \eqref{eq-resTRL}, we have
 \begin{equation}\label{eq-deg2}
 \Primpart_{\{\ox,\oy\}}(\Res_{t_3}(W_{1}^{L},W_{2}^{L}))=M(\oz^L,\ox,\oy,t_1,t_2).
 \end{equation}
Moreover, note that $M$ is non-zero, primitive w.r.t.\ $\{\ox,\oy\}$, and homogeneous
w.r.t.\ $\ot$. Thus,
$\deg_{\ot}(M(\oz^L,\ox,\oy,t_1,t_2))=\deg_{\ot}(M(\oz,\ox,\oy,t_1,t_2))$. Therefore
\[
\begin{array}{rclr}
\deg_{\ot}(\Primpart_{\ox}(\Res_{t_3}(K_{1}^{L},K_{2}^{L})))\!\! &=&\!\! \deg_{\ot}(M(\oz^{L},\cA)) & \text{see \eqref{eq-deg1}} \\
\!\! &=&\!\! \deg_{\ot}(M(\oz,\ox,\oy,t_1,t_2)) & \text{$A_2(\oz^L)\neq 0$} \\
\!\! &=&\!\!\deg_{\ot}(M(\oz,\ox,\oy,t_1,t_2)) & \text{see above} \\
\!\! &=&\!\!\deg_{\ot}(\Primpart_{\ox,\oy}(\Res_{t_3}(W_{1}^{L},W_{2}^{L}))). & \text{see \eqref{eq-deg2}}
\end{array}
\]
So (4) follows. Finally, (5) follows from (4) and the fact that both resultants have the same degree w.r.t.\ $\ot$.
\end{proof}

\para

As a consequence of these lemmas, we get the following theorem  that can be seen as more efficient version of the resultant-based formula in Theorem \ref{theorem-betaW1LW2L}.

\para

\begin{theorem}\label{theorem-BetaP-K1K2}
 Let $\Omega$ be the open set introduced in Lemma \ref{lemm-open-1}.  If $L\in \Omega$, then
\[
\betap=\deg_{\ot}(\content_{\ox}(\Res_{t_3}(K_{1}^{L},K_{2}^{L}))).
\]
\end{theorem}

\begin{proof}
By Theorem  \ref{theorem-betaW1LW2L}, it is enough to prove that
\[
\deg_{\ot}(\content_{\{\ox,\oy\}}(\Res_{t_3}(W_{1}^{L},W_{2}^{L})))=
\deg_{\ot}(\content_{\ox}(\Res_{t_3}(K_{1}^{L},K_{2}^{L}))).
\]
And this is a consequence of Lemma \ref{lemm-open-1}.
\end{proof}

\para

Theorems \ref{theorem-betaW1LW2L} and \ref{theorem-BetaP-K1K2} and Proposition \ref{prop-HSmult} imply the following result about the Hilbert-Samuel multiplicity of the base points.

\para

\begin{corollary}\label{cor-HSW1W2K1K2} Assume the notation of Theorems \ref{theorem-betaW1LW2L} and \ref{theorem-BetaP-K1K2}.  Then:
\begin{enumerate}
\item For every $L \in \myG$, we have
\[
\sum_{A \in \myB(\cP)} e(I_A,R_A) = \deg_{\ot}(\content_{\{\ox,\oy\}}(\Res_{t_3}(W_{1}^{L},W_{2}^{L}))).
\]
\item For every  $L$ in the open set $\Omega$ from Lemma \ref{lemm-open-1}, we have
\[
\sum_{A \in \myB(\cP)} e(I_A,R_A) = \deg_{\ot}(\content_{\ox}(\Res_{t_3}(K_{1}^{L},K_{2}^{L}))).
\]
\end{enumerate}
\end{corollary}

\para

\begin{remark}\label{remark-formulaP}{\rm
Corollary \ref{cor-HSW1W2K1K2} provides the promised resultant-based algorithm to compute the sum of the Hilbert-Samuel multiplicities $e(I_A, R_A)$ for $A\in \myB(\cP)$. }
\end{remark}

\para

\begin{example}\label{example-fromCox} {\rm We consider the surface $\myS$ introduced in \cite{CoxWebPage} parametrized by
\[\cP(\ot):=(p_1:\cdots:p_4)=(t_2^2t_3+t_1^3:t_1^2t_3+t_2^3:t_1t_2t_3:t_2^2t_3).\]    Let us illustrate how to compute $\betap$ by means of resultants. First of all, since
    $(0:0:1)$ belongs to $\myC(p_i)$, we apply a projective transformation. For instance, we replace $\cP(\ot)$ by $\cP(t_1+t_3,t_2+t_3,t_3)$. In this situation, applying Corollary \ref{cor-HSW1W2K1K2}, we see that the sum of the Hilbert-Samuel multiplicities of the base points is given by
\[
\deg_{\ot}(\content_{\{ \ox,\oy\}}(\Res_{t_3}(W_{1}^{L},W_{2}^{L}))) =   \deg_{\ot}(t_{1}^{4}-4t_{1}^{3}t_{2}+6 t_{1}^{2}t_{2}^{2}-4 t_{1}t_{2}^{3}+t_{2}^{4}) = 4.
\]
In fact, this parametrization has $(0:0:1)$ as its unique base point, necessarily of multiplicity 4.  But the above calculation was done without knowing anything about the number of base points or their individual multiplicities.
 } \end{example}

\para

In the next lemma we relate the degree of the primitive part of the resultant to the degree of the surface defined by $\cP$ and by the degree of the rational map induced by $\cP$ (see notation in Section \ref{S-intro}). Note that $\mapdeg(\cP)$ can be computed using \cite{PS-grado}.

\para

\begin{lemma}\label{lemma-formula} There exists a non-empty Zariski open subset $\Omega'$ of $\myGG$ such that for every $L\in \Omega'$, we have
\[ \deg_{\ot}(\Primpart_{\{ \ox\}}(\Res_{t_3}(K_{1}^{L},K_{2}^{L})))=\deg(\myS) \, \mapdeg(\cP), \]
where $\myS$ is the surface parametrized by $\cP$.
\end{lemma}

\begin{proof} We use the notation introduced in the proof of Lemma \ref{lemm-open-1}. In particular, let $\cL=(\cL_1:\cdots:\cL_4)$ be a generic element of $\myGG$.  We construct $\Omega'$ as the intersection of open subsets $\Omega_1,\Omega_2,\Omega_3$.

\medskip

\noindent \textsf{Definition of $\Omega_1$.} This is the subset $\Omega_1$ defined in \eqref{eq-omega1def}.  Recall that for $L\in \Omega_{1}$, we have $\gcd(L_1(\cP),L_4(\cP))=\gcd(L_3(\cP),L_4(\cP))) =1$.

\medskip

\noindent \textsf{Definition of $\Omega_2$.} We want $\Omega_2$ such  that if $L=(L_1:\cdots:L_4)\in \Omega_2\subset \myGG$, then the gradients $\{\nabla (L_1(\cP)(t_1,t_2,1)/L_4(\cP)(t_1,t_2,1)),\,\nabla (L_3(\cP)(t_1,t_2,1)/L_4(\cP)(t_1,t_2,1))\}$ are linearly
independent as vectors in  ${\mathbb  K}(\ot)^2$.  Recall that $p_4\neq 0$ by hypothesis.

Since $\cP$ parametrizes a surface, there exist two different indexes in $\{1,2,3\}$, say w.l.o.g.\ $1$ and $2$, such that $\nabla\left({p_1(t_1,t_2,1)}/{p_4(t_1,t_2,1)}\right)$, $\nabla\left({p_2(t_1,t_2,1)}/{p_4(t_1,t_2,1)}\right)$ are linearly independent.

For $j\in \{1,3\}$, we introduce the gradient vectors
\[ \overline{v_j}(\oz,t_1,t_2)=(v_{j,1},v_{j,2}):=
\nabla\left(\dfrac{\cL_j(\cP)(t_1,t_2,1)}{\cL_4(\cP)(t_1,t_2,1)}\right)=\sum_{i=1}^{3} z_{j,i} \nabla\left(\dfrac{p_i(t_1,t_2,1)}{p_4(t_1,t_2,1)}\right)
\]
as well as the matrix
\[ \Delta=\left(\begin{array}{cc}
v_{1,1}(\oz,t_1,t_2) & v_{1,2}(\oz,t_1,t_2) \\
v_{3,1}(\oz,t_1,t_2) & v_{3,2}(\oz,t_1,t_2)
\end{array} \right).
\]
We observe that $\det(\Delta)\neq 0$ because specializing $\overline{v}_j$ at
$\oz=(1,0,0,0),\overline{0},(0,1,0,0)$ gives
$\nabla\left({p_1(t_1,t_2,1)}/{p_4(t_1,t_2,1)}\right)$ and $
\nabla\left({p_2(t_1,t_2,1)}/{p_4(t_1,t_2,1)}\right)$, which are linearly independent by hypothesis. Let $A_3(\oz)$ be any non-zero coefficient of $\det(\Delta)$ w.r.t.\ $\{t_1,t_2\}$. We define $\Omega_2$ as
\[
\Omega_{2}=\{ L\in \myGG \,|\, A_3(\oz^L)\neq 0\}.
 \]
\smallskip

\noindent \textsf{Definition of $\Omega_3$.} We take $\Omega_3$ as the open subset of $\myGG$ such that if $F(u_1,u_2,u_3,u_4) = 0$ is the implicit equation of the surface parametrized by $\cP$, and $L\in \myGG$,  then $F(L(u_1,u_2,u_3,u_4))$  does not vanish at $(0:1:0:0)$. Note that this means that the total degree, and the partial degree w.r.t.\ $u_2$, of $F(L(u_1,u_2,u_3,u_4))$ are the same.

\medskip

In this situation, we define $\Omega=\Omega_1\cap \Omega_2 \cap \Omega_3$.
We observe that for $L\in \Omega$, the parametrization $L(\cP)$ satisfies the general hypotheses in \cite[p.\ 120]{JSC08}. Namely, no component of $L(\cP)$ vanishes at $(0:0:1)$ and it is $(1,3)$-settled (this terminology is defined in \cite[p.\ 120]{JSC08}). Therefore, by Theorem 6 of \cite{JSC08}, using our notation,
\[
\deg_{x_2}(\myS^L)=\frac{\deg_{\ot}(
\Primpart_{\ox}(\Res_{t_3}(K_{1}^L,K_{2}^{L})))}
{\mapdeg(L(\cP))},
\]
%\deg_{x_2}(\myS^L)=\frac{\deg_{t_3}(S_{1,2})}{\mapdeg(L(\cP))}=
where $\myS^L$ denote the surface parametrized by $L(\cP)$.
Now the result follows by taking into account that since $L\in \Omega_3$, $\deg_{x_2}(\myS)$ is the degree of the surface parametrized by $L(\cP)$, that is,  $\deg_{x_2}(\myS^L)=\deg(\myS)$. Moreover, $\mapdeg(L(\cP))=\mapdeg(\cP)$ since $L\in \myG$.
\end{proof}

As a consequence of the previous lemmas, we have the following degree formula relating degrees and base point locus multiplicity (see notation in Section \ref{S-intro}).

\para

\begin{theorem}\label{theorem-formulaP}
$\betap= \deg(\cP)^2 - \deg(\myS)\cdot \mapdeg(\cP)$.
\end{theorem}

 \begin{proof}
Let $L \in \Omega \cap \Omega'$, where $\Omega$ is from Lemma \ref{lemm-open-1} (and Theorem~\ref{theorem-BetaP-K1K2}) and $\Omega'$ is from Lemma \ref{lemma-formula}.  Since $p_4(0,0,1) \ne 0$ and $L \in \myGG$, we know that $\deg_{t_3}(K_i^L) =\deg(L(\cP))=\deg(\cP)$.  Then
\[
\underbrace{\Res_{t_3}(K_{1}^{L},K_{2}^{L})}_{\begin{array}{l} \text{The degree}\\
\text{is } \deg(\cP)^2 \end{array}}
\ = \ \underbrace{\content_{\ox}(\Res_{t_3}(K_{1}^{L},K_{2}^{L}))}_{\begin{array}{l} \text{By Theorem \ref{theorem-BetaP-K1K2}, the}\\
\text{degree is } \betap \end{array}} \cdot  \underbrace{\Primpart_{\ox}(\Res_{t_3}(K_{1}^{L},K_{2}^{L}))}_{\begin{array}{l} \text{By Lemma \ref{lemma-formula}, the degree}\\
\text{is } \deg(\myS)\hskip1pt \mapdeg(\cP) \end{array}},
\]
where ``degree'' means the degree in $\{t_1,t_2\}$.
 \end{proof}

 \para

 When we combine this theorem with Proposition \ref{prop-HSmult}, we get a new proof of the well-known degree formula (compare with \cite{Cox2001}).

 \para

\begin{corollary}\label{cor-formulaHS}
$\deg(\myS)\cdot \mapdeg(\cP) = \deg(\cP)^2- \sum_{A \in \myB(\cP)} e(I_A,R_A)$.
\end{corollary}

\para

\begin{example}\label{example-verifythm}
{\rm Consider the surface $\myS$ parametrized by
\[\cP(\ot)=(p_1:\cdots:p_4)=(t_2^2t_3+t_1^3:t_1^2t_3+t_2^3:t_1t_2t_3:t_2^2t_3)
\]
from Example \ref{example-fromCox}, where we computed that $\betap = 4$.  One may also check that  $\deg(\cP)=3$, $\deg(\myS)=5$ and $\mapdeg(\cP)=1$ (using results from \cite{PS-grado}). Thus
\[
 \betap=4 = 3^2 - 5\cdot 1 = \deg(\cP)^2- \deg(\myS)\cdot \mapdeg(\cP),
 \]
 as predicted by Theorem \ref{theorem-formulaP}. }
 \end{example}

\para

Applying Theorems \ref{theorem-betaW1LW2L}, \ref{theorem-BetaP-K1K2} and \ref{theorem-formulaP}, and Lemma \ref{lemma-formula},  we get the following resultant-based formula for the degree of the implicit equation of the surface $\myS$.

\para

\begin{theorem}\label{theorem-formulaS}  \
\begin{enumerate}
\item For every $L\in \myG$, we have
$$\begin{array}{lll} \deg(\myS)&= &\dfrac{\deg(\cP)^2-\deg_{\ot}(\content_{\{\ox,\oy\}}(\Res_{t_3}(W_{1}^{L},W_{2}^{L})))}
{\mapdeg(\cP)}.
\\
\noalign{\medskip}
&= &\dfrac{\deg_{\ot}(\Primpart_{\{ \ox\}}(\Res_{t_3}(W_{1}^{L},W_{2}^{L})))}
{\mapdeg(\cP)}.
\end{array}
$$
\item For every $L$ in the open set $\Omega$ introduced in Lemma \ref{lemm-open-1}, we have
 $$\begin{array}{lll}
\deg(\myS)&=&\dfrac{\deg(\cP)^2-\deg_{\ot}(\content_{\{\ox\}}(\Res_{t_3}(K_{1}^{L},K_{2}^{L})))}
{\mapdeg(\cP)} \\
\noalign{\medskip}
&= &\dfrac{\deg_{\ot}(\Primpart_{\{ \ox\}}(\Res_{t_3}(K_{1}^{L},K_{2}^{L})))}
{\mapdeg(\cP)}.
\end{array}$$
\end{enumerate}
\end{theorem}

\para

\begin{remark}\label{rem-hipotesis}{\rm
At the beginning of this section, we imposed two main hypotheses, namely, that $(0:0:1)\not\in \myC(p_i)$ for all $i$ and that $p_4\neq 0$. The first hypothesis was used to relate  $\betap$ with the resultant, and the second was used in Lemma \ref{lemma-formula} to allow the dehomogenization w.r.t.\ the fourth parametrization component. Let us show that the formula in Theorem \ref{theorem-formulaP} is still valid in both cases. If the first hypothesis fails, we can apply a projective transformation $\ell(\ot)$ such that $\cP^*(\ot)=\cP(\ell(\ot))$ satisfies the condition. In this situation, observe that $\deg(\cP^*)=\deg(\cP)$ that $\betap=\mult(\myB(\cP^*))$, and that $\mapdeg(\cP^*)=\mapdeg(\ell)\,\mapdeg(\cP)=\mapdeg(\cP)$. Therefore, since the formula holds for $\cP^*$ it also holds for $\cP$.

On the other hand, if $p_4=0$, we can simply take $L\in \myG$ such that $L(\cP)$ satisfies the hypothesis. Now, the reasoning is as in the previous paragraph.}
\end{remark}

\para

The following corollaries are direct consequences of Theorem \ref{theorem-formulaP}. We observe that Corollary \ref{cor-1} improves the formulae given in Theorem 1 in \cite{SchichoGrado}.

\para

\begin{corollary}\label{cor-1}
$\deg(\cP)\geq \sqrt{\deg(\myS)\,\mapdeg(\cP)} \geq \sqrt{\deg(\myS)}.$
\end{corollary}

\para

\begin{corollary}\label{cor-2}
If $\cP$ is birational, then $\deg(\cP)^2-\betap=\deg(\myS)$.
\end{corollary}

\para

\begin{corollary}\label{cor-3}
A rational surface whose degree is not the square of a natural number cannot be birationally parametrized without base points in $\Bbb{P}^2(\Bbb{K})$.
\end{corollary}

\para

We observe that although the presence of base points might be inevitable (see Corollary \ref{cor-3}), one may reparametrize so that they are all on a line, in particular on the line at infinity (see Theorem 4.1 of \cite{SSV}).

\section{Rational maps of $\proj2$}\label{sec-formula-birational}

In this section, we analyze the base points of rational maps  $\proj2 \dashrightarrow \proj2$ and adapt the results in the previous section to this case. To begin, let
\begin{equation}\label{eq-S}
\begin{array}{cccc}
\cS: & \proj2 &  \dashrightarrow & \proj2   \\
 &\ot=(t_1:t_2:t_3) & \longmapsto & \cS(\ot)=(s_1(\ot):s_2(\ot):s_3(\ot)),
\end{array}
\end{equation}
where  $\gcd(s_1,s_2,s_3)=1$, be a dominant rational transformation of $\proj2$ and let $\mapdeg(\cS)$  denote the degree of the map $\cS$.  Similarly, as in Section \ref{sec-BP-parametrizations}, we \textbf{assume} that $(0:0:1)\not\in \myC(s_i)$ for $i=1,2,3$.
%%%%%%%%%%%%%%%%%%
%Also, let
%\begin{equation}\label{eq-R}
%\begin{array}{cccc}
%\cR: & \proj2 &  \longrightarrow & \proj2   \\
% &\ot=(t_1:t_2:t_3) & \longmapsto & R(\ot)=(r_1(\ot):r_2(\ot):r_3(\ot)),
%\end{array}
%\end{equation}
%where  $\gcd(r_1,r_2,r_3)=1$, be the inverse of $\cS(\ot)$.
Later in Remark \ref{rem-hipotesisS}, we will see that our results hold even when this hypothesis is not satisfied.

\para

\begin{definition}
We say that $A\in \proj2$ is a \textsf{base point} of $\cS(\ot)$ if $s_1(A)=s_2(A)=s_3(A) = 0$. That is, the
base points of $\cS$ are the intersection points of the projective plane curves, $\myC(s_i)$, defined over $\mathbb{K}$ by $s_i(\ot)$, $i=1,2,3$. Let us denote by $\myB(\cS)$ the set of base points of $\cS$; i.e., $\myB(\cS)=\myC(s_1)\cap \myC(s_2) \cap \myC(s_3)$.
\end{definition}

\para

First we introduce the polynomials
\begin{equation}\label{eq-V}
\begin{array}{l}
 V_{1}=\sum_{i=1}^{3} x_i\, s_i(\ot) \in \mathbb{K}(\ox,\oy)[\ot] \\
 \noalign{\medskip}
 V_{2}=\sum_{i=1}^{3} y_i \,s_i(\ot)\in
 \mathbb{K}(\ox,\oy)[\ot],
 \end{array}
\end{equation}
where $x_i,y_j$ are new variables; compare with \eqref{eq-W}. Then, as we did in Section \ref{sec-BP-parametrizations}, we have the following notion of multiplicity.

\para

\begin{definition}
For $A\in \myB(\cS)$, we define the \textsf{multiplicity of intersection of $A$} as $\mult_{A}(\myC(V_{1}),\myC(V_2))$.

In addition, we define the \textsf{multiplicity of the base points locus of ${\cS}$}, denoted $\betas$, as
\[
\betas:=\sum_{A\in \myB({\cS})} \mult_A(\myC(V_{1}),\myC(V_{2}))
\]
\end{definition}

For every $L\in \myGS$ (see the notation in Section \ref{S-intro}) we introduce the polynomials (compare with \eqref{eq-WL})
\begin{equation}\label{eq-VL}
\begin{array}{l}
 V_{1}^{L}=\sum_{i=1}^{3} x_i\,L_i(\cS) \in \mathbb{K}(\ox,\oy)[\ot] \\
 \noalign{\medskip}
 V_{2}^{L}=\sum_{i=1}^{3} y_i \,L_i(\cS)\in
 \mathbb{K}(\ox,\oy)[\ot],
 \end{array}
\end{equation}

In this situation,  Proposition \ref{Prop-BasePointS} extends naturally to the case of the map $\cS$, and hence the following theorem holds (compare with Theorem \ref{theorem-betaW1LW2L}).

\para

\begin{theorem}\label{theorem-betaV1LV2L} If $L\in \myGS$, then
$$\betas=\deg_{\ot}(\content_{\{\ox,\oy\}}(\Res_{t_3}(V_{1}^{L},V_{2}^{L}))). $$
\end{theorem}

\para

For $L\in \myGS$, we consider the polynomials (compare with \eqref{eq-K1K2g})
\begin{equation}\label{eq-J1J2}
\begin{array}{lll}
J_{1}^{L}(\ox,\ot) &=& V_{1}^{L}(x_3,0,-x_1,\ot)=x_{3} L_1(\cS)-x_1 L_3(\cS)\in \mathbb{K}(\ox)[\ot],\\
\noalign{\medskip}
J_{2}^{L}(\ox,\ot) &=& V_{2}^{L}(0,x_3,-x_2,\ot)=x_3L_2(\cS)-x_2L_3(\cS) \in \mathbb{K}(\ox)[\ot].
\end{array}
\end{equation}
Similar to Section \ref{sec-formula-birational}, the corresponding versions of  Lemma \ref{lemm-open-1} and Theorem \ref{theorem-BetaP-K1K2} hold. We state here the version of Theorem \ref{theorem-BetaP-K1K2} for $\cS$.

\para

\begin{theorem}\label{theorem-BetaP-J1J2}
There exists a non-empty open subset $\Omega_{\cS}$ of $\myGS$ such that for  $L\in \Omega_{\cS}$, we have
\[ \betas=\deg_{\ot}(\content_{\ox}(\Res_{t_3}(J_{1}^{L},J_{2}^{L}))). \]
\end{theorem}

\para

The results in the last part of Section \ref{sec-BP-parametrizations} involve  surface parametrizations in $\Bbb{P}^3(\Bbb{K})$. In order to apply these results to a map $\cS$ as in \eqref{eq-S}, we consider the map
\begin{equation}\label{eq-ParamAssToS}
\begin{array}{llll}
\cP^{\cS}: & \mathbb{P}^{2}(\K) &\dashrightarrow &\myS^{\cS} \subset \mathbb{P}^{3}(\K) \\
& \ot & \longmapsto & (s_1(\ot):s_2(\ot):s_2(\ot):s_3(\ot))
\end{array}
\end{equation}
We observe that the rank of the Jacobian of $\cS$ is 2, and hence the rank of the Jacobian of $\cP^{\cS}$ is also 2. Therefore, $\myS^{\cS}$ is a surface. Moreover, since $s_i(0,0,1)\neq 0$ for all $i\in \{1,2,3\}$, none of the curves defined by the components of $\cP^{\cS}$ passes through $(0:0:1)$ either. Also note that $\myS^{\cS}$ is not the plane $u_4=0$; rather, $\myS^{\cS}$ is the plane $u_2=u_3$. So $\cP^{\cS}$ satisfies the hypotheses required in Section \ref{sec-BP-parametrizations}. In addition, we clearly have  $\mapdeg(\cS)=\mapdeg(\cP^\cS)$.
%since $\cS$ is birational, $\cP$ is birational too, and hence $\mapdeg(\cP^\cS)=1$.

\para

Next lemma relates the multiplicities of the base point loci $\myB(\cS)$ and $\myB(\cP^{\cS})$.

\para

\begin{lemma}\label{lemma:betaPbetaS}
$\myB(\cS) = \myB(\cP^\cS)$ and $\betas=\mult(\myB(\cP^\cS))$.
\end{lemma}
\begin{proof}
The first assertion is obvious since $\cS = (s_1,s_2,s_3)$ and $\cP^\cS =  (s_1,s_2,s_2,s_3)$.  For the second, first note that the analog of Proposition \ref{prop-HSmult} holds for $\cS$, so that $\betas$ is the sum of the Hilbert-Samuel multiplicities of the base points for the ideal generated $\cS$.  Since $\cS$ and $\cP^\cS$ give the same ideal, this equals the sum of the Hilbert-Samuel multiplicities of the base points for the ideal generated by $\cP^\cS$. Hence the sum is $\mult(\myB(\cP^\cS))$ by Proposition \ref{prop-HSmult}.
\end{proof}

\para

In this situation, we can adapt Lemma \ref{lemma-formula} and Theorem \ref{theorem-formulaP} to the case of the map $\cS$ as follows.

\para

\begin{theorem}\label{theorem-formulapbS} \
\begin{enumerate}
\item $\betas=\deg(\cS)^{2}-\mapdeg(\cS)$.
\item There exists a non-empty Zariski open subset $\Omega'_\cS$ of $\myGS$ such that for every $L\in \Omega'_\cS$, we have
\[
\deg_{\ot}(\Primpart_{\{ \ox\}}(\Res_{t_3}(J_{1}^{L},J_{2}^{L})))=\mapdeg(\cS)
\]
\end{enumerate}
\end{theorem}

\begin{proof}
Observe that $\deg(\cP^\cS) = \deg(\cS)$ and $\mapdeg(\cP^\cS) = \mapdeg(\cS)$.  Since $\cP^\cS$ parametrizes the plane $u_2 = u_3$ in $\Bbb{P}^3$, the image surface $\myS^\cS$ has $\deg(\myS^\cS) = 1$. Hence
\[
\begin{array}{rclr}
\mapdeg(\cS) &= & 1 \cdot \mapdeg(\cP^{\cS}) &\\
&= & \deg(\cP^{\cS})^2-\mult(\myB(\cP^\cS)) & \text{(see Theorem \ref{theorem-formulaP})}\phantom{.}\\
&= & \deg(\cS)^2-\betas & \text{(see Lemma \ref{lemma:betaPbetaS})}.
\end{array}
\]
This proves statement (1).

For (2), assume for the moment that we have a non-empty open subset $\Omega_4$ of $\myGS$ such that $\deg_{\ot}(J_1^L) = \deg_{\ot}(J_2^L) = \deg(\cS)$ for all $L \in \Omega_4$.  Set $\Omega'_\cS = \Omega_4 \cap \Omega_\cS$, where $\Omega_\cS$ is from Theorem~\ref{theorem-BetaP-J1J2}.

Now take $L \in \Omega'_\cS$. Then Theorem \ref{theorem-BetaP-J1J2} allows us to rewrite statement (1) in the form
\[
\mapdeg(\cS) = \deg(\cS)^2 - \deg_{\ot}(\content_{\{\ox,\oy\}}(\Res_{t_3}(J_{1}^{L},J_{2}^{L}))).
\]
However, we have the factorization
\[
\Res_{t_3}(J_{1}^{L},J_{2}^{L}) = \content_{\{\ox,\oy\}}(\Res_{t_3}(J_{1}^{L},J_{2}^{L})) \cdot \Primpart_{\{ \ox\}}(\Res_{t_3}(J_{1}^{L},J_{2}^{L})).
\]
This resultant has degree $\deg(\cS)^2$ w.r.t.\ $\ot$ since $L \in \Omega_4$, and statement (2) follows.

It remains to construct $\Omega_4$.  Let $\cL=(\cL_1:\cL_2:\cL_3)$ be a generic projective transformation; that is, $\cL_i=z_{i,1} t_1+ z_{i,2} t_2+ z_{i,3} t_3$, where $z_{i,j}$ are undetermined coefficients satisfying that the determinant of the corresponding matrix is not zero.  Since $s_i(0,0,1) \ne 0$ for all $i$, arguing as in the proof of Lemma \ref{lemm-open-1}, we obtain
\begin{align*}
\LC_{t_3}(J_1^\cL) &= x_3 \sum_{j=1}^3 z_{1,j} \LC_{t_3}(s_i) - x_1 \sum_{j=1}^3 z_{3,j} \LC_{t_3}(s_i)\\
\LC_{t_3}(J_2^\cL) &= x_3 \sum_{j=1}^3 z_{2,j} \LC_{t_3}(s_i) - x_2 \sum_{j=1}^3 z_{3,j} \LC_{t_3}(s_i)
\end{align*}
(compare with \eqref{eq-LCoef}).  If we set $A_1(\oz) := \sum_{j=1}^3 z_{3,j} \LC_{t_3}(s_i)$, then the desired $\Omega_4$ consists of all $L \in \myGS$ such that $A_1(\oz^L) \ne 0$ ($\oz^L$ is defined in the proof of Lemma \ref{lemm-open-1}).
\end{proof}

Applying Theorems \ref{theorem-betaV1LV2L}, \ref{theorem-BetaP-J1J2} and \ref{theorem-formulapbS},  we get the following resultant-based formula which is the corresponding version of Theorem \ref{theorem-formulaS}.

\para

\begin{theorem}\label{theorem-formulaSbis}  \
\begin{enumerate}
\item If $L\in \myGS$, then
$$\begin{array}{lll} 1&= &\dfrac{\deg(\cS)^2-\deg_{\ot}(\content_{\{\ox,\oy\}}(\Res_{t_3}(V_{1}^{L},V_{2}^{L})))}
{\mapdeg(\cS)}.
\\
\noalign{\medskip}
&= &\dfrac{\deg_{\ot}(\Primpart_{\{ \ox\}}(\Res_{t_3}(V_{1}^{L},V_{2}^{L})))}
{\mapdeg(\cS)}.
\end{array}
$$
\item For every $L$ in the open set $\Omega'_\cS$ defined in Theorem \ref{theorem-formulapbS}, we have
 $$\begin{array}{lll}
1&=&\dfrac{\deg(\cS)^2-\deg_{\ot}(\content_{\{\ox\}}(\Res_{t_3}(J_{1}^{L},J_{2}^{L})))}
{\mapdeg(\cS)} \\
\noalign{\medskip}
&= &\dfrac{\deg_{\ot}(\Primpart_{\{ \ox\}}(\Res_{t_3}(J_{1}^{L},J_{2}^{L})))}
{\mapdeg(\cS)}.
\end{array}$$
\end{enumerate}
\end{theorem}

\para

Since a birational map of $\proj2$ has $\mapdeg(\cS)=1$, we get the following corollaries.

\para
\begin{corollary}\label{C-formulaSbis} If $\cS$ is birational, then
$\betas=\deg(\cS)^{2}-1.$
\end{corollary}

\para

\begin{corollary}
Every non-linear birational transformation of $\proj2$ has base points.
\end{corollary}

\para

\begin{remark}\label{rem-hipotesisS}{\rm
At the beginning of this section, we required that $(0:0:1)\not\in \myC(s_i)$ for $i=1,2,3$. Reasoning as in Remark \ref{rem-hipotesis}, we get that since the formula holds for $\cS^*(\ot)=\cS(\ell(\ot))$ ($\ell(\ot)$ is a projective transformation), it also holds for $\cS$.}
\end{remark}

\para

In the last part of this section, we discuss an additional property satisfied by birational transformations of $\proj2$. This property is related with the rationality of the curves $\myC(J_{1}^{L})$  and $\myC(J_{2}^{L})$.
\para

\begin{lemma}\label{lemm-open-2}
There exists a non-empty Zariski open subset $\Omega''_\cS$ of $\myGS$ such that for every $L\in \Omega$, $\myC(J_{i}^{L})$ is irreducible.
\end{lemma}
\begin{proof}
Let $\cL=(\cL_1:\cL_2:\cL_3)$ be a generic projective transformation as in the proof of Theorem \ref{theorem-formulapbS} and set $\overline{z}_i=(z_{i,1},z_{i,2},z_{i,3})$. Let $A_i(\overline{z}_i)$ be the leading coefficient of $\cL_i(S)$ w.r.t.\ $t_3$. Now set $R^\cL(\overline{z}_1,\overline{z}_3,t_1,t_2):=\Res_{t_3}(\cL_1(S),\cL_3(S))\in \mathbb{K}[\overline{z}_1,\overline{z}_3,t_1,t_2]$. If $R^\cL=0$, then $\cL_1(S),\cL_3(S)$ have a common factor $B$. Arguing as in the proof of Lemma \ref{lemm-open-1},  $B\in \mathbb{K}[\ot]$. So, in particular, $B$ divides $\cL_1(S)$, and hence $B$ divides $\gcd(s_1,s_2,s_3)$, a contradiction. Therefore $R^\cL$ is non-zero.  Then let  $M_1(\overline{z}_1,\overline{z}_3)$ be the gcd of all coefficients of $R^\cL$ w.r.t.\ $\{t_1,t_2\}$. Repeating the same argument for $\cL_2(S)$ and $\cL_3(S)$, we get a polynomial $M_2(\overline{z}_{2},\overline{z}_{3})$.

In this situation, let $\Omega''_\cS$ consists of all projective transformations whose coefficients are not zeros of $A_1\cdot A_2 \cdot A_3 \cdot M_1\cdot M_2$.  If $L \in \Omega''_\S$, then
 $J_{1}^{L}$ and $J_{2}^{L}$ are irreducible. Indeed, if $J_{1}^{L}$ is reducible, then $\gcd(L_1(S),L_3(S))\neq 1$. Moreover, since
 $A_1$ and $A_3$ do not vanish, $\Res_{t_3}(L_1(S),L_3(S))$ specializes properly. Thus, $R^\cL(\overline{z}_1,\overline{z}_3,t_1,t_2)$ vanishes, and hence $M_1$ also vanishes, a contradiction.  Similar reasoning shows that
 $J_{2}^{L}$ is also irreducible.
\end{proof}

\para

\begin{example}\label{example-section3} {\rm
Consider the classical Cremona transform $\cS(\ot) = (t_2t_3,t_1t_3,t_1t_2)$.  It has base points $\{(1:0:0), (0:1:0), (0:0:1)\}$ and $\deg(\cS) = 2$.  Since $\cS$ is birational, Theorem \ref{theorem-formulapbS} implies
\[
\betas = 2^2 - 1 = 3.
\]
Hence base point has multiplicity 1.  Also notice that the polynomials
\begin{align*}
J_1(\ox,\ot) &= x_3(t_2t_3) - x_1(t_1t_2) = t_2(x_3 t_3 - x_1 t_1)\\
J_2(\ox,\ot) &= x_3(t_1t_3) - x_2(t_1t_2) = t_1(x_3 t_3 - x_2 t_2)
\end{align*}
are not irreducible.  This explains why the open set $\Omega''_\cS$ is needed in Lemma \ref{lemm-open-2}.}
\end{example}

\para

\begin{proposition}\label{prop-K1K2} Let $\cS$  be a  birational map  of $\proj2$ and  $\Omega''_\cS$ be the open subset from Lemma \ref{lemm-open-2}.  Assume $L \in \Omega''_\cS$ and let  $\cR^{L}=\cR\circ L^{-1}=(r_{1}^{L}:r_{2}^L:r_{3}^{L})$  be the inverse of $L\circ \cS$.   Then we have:
\begin{enumerate}
\item $\myC(J_{1}^{L})$ is rational  and can be parametrized by
\[ \cJ_1(h_1,h_2)=(j_{1,1}(x_1,x_3,h_1,h_2):j_{1,2}(x_1,x_3,h_1,h_2):j_{1,3}(x_1,x_3,h_1,h_2)),  \]
where  $j_{1,i}(x_1,x_3,h_1,h_2)$ is the homogenization of $r_{i}^{L}(x_1,h_1,x_3)$ as polynomial in $\mathbb{K}[\ox][h_1]$.
\item $\myC(J_2^{L})$ is rational   and can be parametrized by
\[ \cJ_2(h_1,h_2)=(j_{2,1}(x_2,x_3,h_1,h_2):j_{2,2}(x_2,x_3,h_1,h_2):j_{2,3}(x_2,x_3,h_1,h_2)),  \]
where $j_{2,i}(x_2,x_3,h_1,h_2)$ is the homogenization of $r_{i}^{L}(h_1,x_2,x_3)$ as polynomial in $\mathbb{K}[\ox][h_1]$.
\end{enumerate}
\end{proposition}
\begin{proof}  Since the $J_{i}^{L}$ are irreducible polynomials (see Lemma \ref{lemm-open-2}) and $\cR^L$ is the inverse of $L\circ \cS$,  we have  $J_{i}^{L}(\cJ_i(h_1,h_2))=0$. This proves (1) and (2).
\end{proof}

\para

A natural question is whether the curves $\myC(K_{i}^{L})$ in $\Bbb{P}^3(\Bbb{F})$ (see \eqref{eq-K1K2g}), when $\cP$ is birational, are also rational. However, in general, this is not true. For instance, consider
\[
\cP(\ot)=(t_1^3-t_1 t_2 t_3-t_3^3: t_2 t_3^2-t_1^3-5 t_3^3: t_1^3-t_2^2 t_3-t_1^2 t_3+4 t_3^3: t_1^3-t_2 t_3^2-t_3^3).
\]
One may check that $\mapdeg(\cP)=1$ (use \cite{PS-grado}), $\betap=3$ and $\deg(\myS)=6$ (check that the formula in Theorem \ref{theorem-formulaP} holds). However, there exists a non-empty open Zariski subset $\Omega$ of $\myG$ such that for every $L\in \Omega$, the curves $\myC(K_{1}^{L})$ and $\myC(K_{2}^{L})$ have genus 1.

\section{Behaviour of base points under composition}\label{sec-composition}

In this section, we analyze the relation between the base loci of two different parametrizations  of the same surface under the assumption that one is the reparametrization of the other. More precisely, in the sequel we fix a surface $\myS\subset \projtres$, as well as two rational parametrization of $\myS$, namely $\cP$ and $\cQ$. Moreover, we \textbf{assume} that there exists a rational map $\cS$ of $\proj2$ such that $\cP=\cQ\circ \cS$.  Note that if $\cQ$ is birational then $\cS$ always exists; indeed, in that case, $\cS=\cQ^{-1}\circ \cP$.
In this situation, our goal is to relate $\betap, \betas$, and $\betaq$.

 \para

To begin, let  $\cQ(\ot)=(q_1 :\cdots:q_4)$, $\cS(\ot)=(s_1 :s_2:s_3)$ where $\gcd(q_1,\ldots,q_4) =\gcd(s_1,s_2,s_3)=1$. Also set  $p_i(\ot)=q_i(s_1(\ot),s_2(\ot),s_3(\ot))$.  Here is a first result.

\para

\begin{proposition}\label{prop-ineq}\
\begin{enumerate}
\item $\deg(\cP)\le \deg(\cQ) \,\deg(\cS)$.
\item $\betap\leq\deg(\cS)^2 \betaq+\deg(\myS)\, \mapdeg(\cQ) \,\betas$.
\end{enumerate}
\end{proposition}

\begin{proof}
For (1), note that $\deg(p_i) = \deg(\cQ) \,\deg(\cS)$.  Then the desired inequality follows since $\cP$ is obtained from the $p_i$ after dividing out by $\gcd(p_1,p_2,p_3,p_4)$.

For (2), Theorems \ref{theorem-formulaP} and \ref{theorem-formulapbS} imply
\begin{align*}
\deg(\cP)^2 &= \betap+\deg(\myS)\,  \mapdeg(\cP)\\
\deg(\cQ)^2 \deg(\cS)^2 &=  \big(\betaq+\deg(\myS)\, \mapdeg(\cQ)\big) \big(\betas+\mapdeg(\cS)\big).
\end{align*}
Since $\mapdeg(\cP)=\mapdeg(\cQ) \,\mapdeg(\cS)$, it follows that
\begin{equation}\label{eq-degbeta}
\begin{array}{c}
\begin{aligned}
\deg(\cP)^2 -\deg(\cQ)^2 \deg(\cS)^2
&= \betap - \big(\betaq\, \betas\,+ \\
&\quad \ \betaq \, \mapdeg(\cS)\, + \\
&\quad \ \deg(\myS) \, \mapdeg(\cQ) \, \betas\big)\\
&=\betap - \big(\deg(\cS)^2 \betaq\, + \\ &\quad \ \deg(\myS)\, \mapdeg(\cQ) \,\betas\big),
\end{aligned}
\end{array}
\end{equation}
where the last equality uses $\deg(\cS)^2 = \betas + \mapdeg(\cS)$ by Theorem \ref{theorem-formulapbS}.  By (1), the left-hand side is non-positive, so the same is true for the right-hand side.
\end{proof}

In the following theorem, we characterize when the inequalities in Proposition \ref{prop-ineq} are equalities.
\para

\begin{theorem}\label{T-propertyP1}  The following statements are equivalent:
\begin{enumerate}
   \item   $\gcd(p_1,p_2,p_3,p_4)=1$.
  \item    $\deg(\cP)=\deg(\cQ) \,\deg(\cS)$.
   \item   $\betap=\deg(\cS)^2 \betaq+\deg(\myS)\, \mapdeg(\cQ) \,\betas$.
   \end{enumerate}
\end{theorem}

\begin{proof}
$1 \Leftrightarrow 2$. This follows from the proof of statement (1) of Proposition \ref{prop-ineq}.

\medskip

\noindent $2 \Leftrightarrow 3$. This is an immediate consequence of \eqref{eq-degbeta}.
\end{proof}

\para

The following corollary follows directly from the previous result.

\para

\begin{corollary}\label{cor-betapqs}
If $\gcd(p_1,\ldots,p_4)=1$, then we have:
\begin{enumerate}
\item $\myB(\cQ)=\emptyset$ if and only if $\betap=\deg(\myS)\, \mapdeg(\cQ) \, \betas$.
\item If $\myB(\cQ)=\emptyset$ and $\cQ$ is birational, then $\betap=\deg(\myS)\, \betas$.
\item $\myB(\cP)=\emptyset$ if and only if $\myB(\cQ)=\emptyset=\myB(\cS).$
\end{enumerate}
\end{corollary}

\para

\begin{theorem}\label{T-propertyP0}  If  $\myB(\cQ)=\emptyset$,  then $\gcd(p_1,p_2,p_3,p_4)=1$.
 \end{theorem}
\begin{proof}
 Assume that a non-constant polynomial $h(\ot)\in {\Bbb K}[\ot]$ divides $p_i$ for all $i$.  Then $h$ divides $q_i(s_1,s_2,s_3)$ for all $i$, so that for each $\mathbf{a}\in \myC(h)$, $q_i(s_1(\mathbf{a}),s_2(\mathbf{a}),s_3(\mathbf{a}))=0$ for all $i$.
But $\myC(s_1)\cap \myC(s_2)\cap \myC(s_3)$ is finite since $\gcd(s_1,s_2,s_3)=1$.  It follows that $\myC(h)\setminus (\myC(s_1)\cap \myC(s_2)\cap \myC(s_3))\neq \emptyset. $ Let $\mathbf{a}\in \myC(h)\setminus (\myC(s_1)\cap \myC(s_2)\cap \myC(s_3))$. Then $(s_1(\mathbf{a}),s_2(\mathbf{a}),s_3(\mathbf{a}))\in \proj2$ and hence is a base point of $\cQ$, a contradiction.
\end{proof}

\para

Theorems \ref{T-propertyP1} and \ref{T-propertyP0} have the following nice corollary.

\para

\begin{corollary}\label{cor-noBPQ} If $\myB(\cQ)=\emptyset$, then  $\deg(\cP) =\deg(\cQ)\,\deg(\cS)$.
\end{corollary}

\para

\begin{remark}\label{R-propertyP1} \
{\rm
%Using Theorem \ref{T-propertyP1}, we see that  the degree of the composition decreases, i.e., $\deg(\cP) <\deg(\cQ)\,\deg(\cS)$, if and only if $\gcd(p_1,p_2,p_3,p_4)=h\in {\Bbb K}[\ot]\setminus \mathbb{K}$.  Moreover, in this case, $\cS(\myC(h))\subset \myB(\cQ)$. Therefore, as a conclusion, we see that if $\cQ$ does not have base points, then $\deg(\cP) = \deg(\cQ)\,\deg(\cS)$ (see Theorem \ref{T-propertyP0} and \ref{T-propertyP1}).
%Using Theorem \ref{T-propertyP1}, we get that $\deg(\cP)< \deg(\cQ) \cdot \deg(\cS)$ if and only if   $\gcd(p_1,p_2,p_3,p_4)=h\in {\Bbb K}[\ot]\setminus \mathbb{K}$  (which implies that $\myB(\cQ)\neq \emptyset$ by Theorem    \ref{T-propertyP0}). Moreover, in that case, $\cS(\myC(h))\subset \myB(\cQ)$.
    %% Then if we remove the factor $h(\ot)$ we eliminate the base point of $\cQ(\ot)$,  $(a:b:c)$, and also the base points of $S(\ot)$,  $\ot_0$, where $k(\ot_0)=0$.
The converse of Theorem \ref{T-propertyP0} is not true.}
%Also the reciprocals of Corollary \ref{cor-betapqs} (1) and (2) are not true.
%\end{enumerate}}
\end{remark}

\para

\begin{example}\label{example-section4} {\rm
Consider the parametrization from Examples \ref{example-fromCox} and \ref{example-verifythm}, which we write as
\[
\cQ(\ot) =(t_2^2t_3+t_1^3:t_1^2t_3+t_2^3:t_1t_2t_3:t_2^2t_3).
\]
We know that $\deg(\cQ) = 3, \mapdeg(\cQ) = 1, \deg(\myS) = 5$ and $\betaq = 4$.  If $\cS(\ot) = (t_2t_2,t_1t_3,t_1t_2)$ is the Cremona transform from Example \ref{example-section3}, then the reparametrization $\cP = \cQ \circ \cS$ is given by
\begin{align*}
\cP(\ot) &= (t_1^3t_2(t_3^2+t_2^2):t_1t_3^2(t_2^3+t_1^2t_3):t_1^2t_2^2t_3^2:t_1^3t_2t_3^2)\\
&= (t_1^2t_2(t_3^2+t_2^2):t_3^2(t_2^3+t_1^2t_3):t_1t_2^2t_3^2:t_1^2t_2t_3^2),
\end{align*}
where the second line factors out the common factor $t_1$.  Thus $\deg(\cP) = 5 < 3\cdot 2 = \deg(\cQ) \,\deg(\cS)$.  Furthermore,
\begin{align*}
&\betap = \deg(\cP)^2 - \deg(\myS) \, \mapdeg(P) = 5^2 - 5 \cdot 1 = 20\\
&\deg(\cS)^2 \betaq+\deg(\myS)\, \mapdeg(\cQ) \,\betas =
2^2 \cdot 4 + 5 \cdot 1 \cdot 3 = 31.
\end{align*}
This shows that when $\myB(\cQ) \ne \emptyset$, the inequalities in Proposition \ref{prop-ineq} can be strict.}
\end{example}

\para

The next theorem extends Corollary \ref{cor-betapqs} (1) using the curves from \eqref{eq-K1K2g} and \eqref{eq-J1J2}.

\para

\begin{theorem}\label{T-propertyP2} If  $\myB(\cQ)=\emptyset$, then  $\betap=\deg(\myS) \, \mapdeg(\cQ) \, \betas$. Furthermore,
\[\Content_{\{ \ox\}}(\Res_{t_3}(K_{1}^{L_\cP},K_{2}^{L_\cP}))=\Content_{\{ \ox\}}(\Res_{t_3}(J_{1}^{L_\cS}, J_{2}^{L_\cS}))^{\deg(\myS) \, \mapdeg(\cQ)}.\]
       where
$L_\cP$ belongs to the open set introduced in Theorem \ref{theorem-BetaP-K1K2} and $L_\cS$ belongs to the open set introduced in Theorem \ref{theorem-BetaP-J1J2}.
\end{theorem}
\begin{proof}
% By Theorems \ref{T-propertyP1} and \ref{T-propertyP0}, we have that $\deg(\cP)=\deg(\cQ)\cdot \deg(\cS)$ and $\betap=\deg(\cS)^2 \cdot \betaq+\deg(\myS)\cdot \mapdeg(\cQ) \cdot \betas$. Since $\betaq=0$, we get that $\betap=\deg(\myS) \cdot \mapdeg(\cQ) \cdot \betas$. In addition, it holds that  $R_{\cP}=R_\cS^{\deg(\myS) \cdot \mapdeg(\cQ)}$, where
By Corollary \ref{cor-betapqs} (1), $\betap=\deg(\myS) \, \mapdeg(\cQ) \, \betas$. Now let us prove that $R_{\cP}=R_\cS^{\deg(\myS) \, \mapdeg(\cQ)}$, where
\[
\begin{array}{l} R_{\cP}(t_1, t_2):=\Content_{\{ \ox\}}(\Res_{t_3}(K_{1}^{L_\cP},K_{2}^{L_\cP})) \\
\noalign{\medskip}
R_\cS(t_1, t_2):=\Content_{\{ \ox\}}(\Res_{t_3}(J_{1}^{L_\cS}, J_{2}^{L_\cS})).
\end{array}
\]
       Indeed, by Theorems \ref{theorem-BetaP-K1K2} and Theorem \ref{theorem-BetaP-J1J2}, we know that
       $\deg_{\ot}(R_\cP)=\betap$ and $\deg_{\ot}(R_\cS)=\betas$. On the other hand, recall that $p_i=q_i(s_1,s_2,s_3)$ for $i=1,\ldots,4$. Reasoning as in the proof of Theorem \ref{T-propertyP0}, we  see that every base point of $\cP$ is a base point of $\cS$ (remember that $\betaq=0$). Furthermore, it is clear that every base point of $\cS$ is a base point of $\cP$ ($q_i(0,0,0)=0$). Then $R_{\cP}=R_\cS^\alpha$ for some exponent $\alpha$,  and since  $\betap=\deg(\myS) \, \mapdeg(\cQ)\, \betas$, we conclude that $R_{\cP}=R_\cS^{\deg(\myS) \, \mapdeg(\cQ)}$.
\end{proof}

%\para

%\begin{remark}\label{R-propertyP2} Note that the reciprocal of  Theorem \ref{T-propertyP2} is not true.
%\textcolor{red}{HABLAR}
%\end{remark}

\appendix

\section{Some underlying algebra and geometry}

In this appendix, we discuss the algebra and geometry behind Theorem \ref{theorem-formulaP}, which we write in the form
\begin{equation}\label{eq-degformula}
\deg(\cP)^2 = \betap + \deg(\myS)\cdot \mapdeg(\cP).
\end{equation}
Our approach in this appendix, based on \cite{CoxWebPage}, is intuitive and non-rigorous.

The polynomials $W_1$ and $W_2$ defined in \eqref{eq-W} are linear combinations of the parametrization $\cP = (p_1,\dots,p_4)$ with coefficients given by new variables $x_1,\dots,x_4$ and $y_1,\dots,y_4$.  For the time being, we will regard the $x_i$ and $y_i$ as generic elements of the base field $\Bbb{K}$.  Later in the discussion, they will resume their role as independent variables.

With this convention, $W_1$ and $W_2$ define curves in $\Bbb{P}^2(\Bbb{K})$.  By B\'ezout's Theorem, their points of intersection, counted with multiplicity, add up to $\deg(\cP)^2$.  This is the left-hand side of \eqref{eq-degformula}.

Intersection points of the curves $\myC(W_1)$ and $\myC(W_2)$ come in two flavours:
\begin{itemize}
\item The $p_i$ all vanish at the base points $\myB(\cP)$, so the same is true for $W_1$ and $W_2$.  Hence $\myC(W_1)$ and $\myC(W_2)$ always intersect at the base points.  These are always the same, no matter how we choose $x_i$ and $y_i$.
\item The remaining points in $\myC(W_1)\cap\myC(W_2)$ depend on the choice of $x_i$ and $y_i$.
\end{itemize}
Let us explain how these two flavours contribute the right-hand side of \eqref{eq-degformula}:
\begin{itemize}
\item In the notation of Proposition \ref{prop-HSmult}, a base point $A \in \myB(\cP)$ contributes $\dim_{\Bbb{K}} R_{A}/\langle \widetilde{W}_1,\widetilde{W}_2\rangle$ to B\'ezout's Theorem.  As noted in the proof, this equals the Hilbert-Samuel multiplicity $e(I_A,R_A)$.  Summing these up, we see that the base points contribute $\betap$ to  B\'ezout's Theorem, which explains the first summand on the right-hand side of \eqref{eq-degformula}.

The key point here is that to compute $e(I_A,R_A)$, we replace $I_A$ with a reduction ideal (see \cite[4.6]{BH}).  Since $R_A$ has dimension two, the reduction ideal is generated by two generic linear combinations of the generators of $I_A$.  From the point of view of commutative algebra, this explains why we work with $W_1$ and $W_2$.
\item For the remaining points of intersection, consider the surface $\myS$ parametrized by $\cP$.  Its degree $\deg(\myS)$ is the number of points where a generic line intersects $\myS$.  This line is the intersection of two generic planes $H_1$ and $H_2$.  For homogeneous coordinates $u_1,\dots,u_4$ of $\Bbb{P}^3(\Bbb{K})$, we can let $H_1 = \myC(\sum_{i=1}^4 x_i u_i)$ and $H_2 = \myC(\sum_{i=1}^4 y_i u_i)$ since $x_i$ and $y_i$ are generic.  Via the parametrization $\cP$, the curves $H_1 \cap \myS$ and $H_2 \cap \myS$ on $\myS$ pull back to $\myC(W_1)$ and $\myC(W_2)$ in $\Bbb{P}^2(\Bbb{K})$.  From the point of view of geometry, this explains why we work with $W_1$ and $W_2$.

Since $H_1 \cap H_2$ is generic, we can assume that $H_1 \cap H_2$ meets $\myS$ transversely at $\deg(\myS)$ smooth points of $\myS$ and that $\mapdeg(\cP)$ points of $\Bbb{P}^2(\Bbb{K})$ map to each point of $H_1 \cap H_2 \cap \myS$.  This gives $\deg(\myS)\cdot \mapdeg(\cP)$ points of $\Bbb{P}^2(\Bbb{K})$, all contained in $\myC(W_1)\cap\myC(W_2)$ by our choice of $H_1$ and $H_2$.  Genericity implies that $\cP$ is \'etale at these points (i.e., the Jacobian has maximal rank).  When combined with transversality, it follows that each point contributes 1 to B\'ezout's Theorem.  This explains the second summand on the right-hand side of \eqref{eq-degformula}.
\end{itemize}

It remains to explain how this relates to the resultants that appear in the body of the paper.   We begin with the proof of B\'ezout's Theorem from \cite[Chapter 8, \secsym 7]{Cox2015}.  A coordinate change $L \in \myG$ gives the polynomials $W_{1}^{L},W_{2}^{L}$ from \eqref{eq-WL}.  The basic idea of the proof is that if $L$ is sufficiently generic, then the resultant $\Res_{t_3}(W_{1}^{L},W_{2}^{L})$ is a homogeneous polynomial in $t_1,t_2$ whose irreducible factors correspond to the points of intersection and whose exponents give the corresponding multiplicities.  Since the resultant has degree $\deg(\cP)^2$, this proves B\'ezout's Theorem.

So far, $\ox = (x_1,\dots,x_4)$ and $\oy = (y_1,\dots,y_4)$ have been generic elements of $\Bbb{K}$.  But now let them return to being independent variables.  Then $\Res_{t_3}(W_{1}^{L},W_{2}^{L})$ is a polynomial in $t_1,t_2,\ox,\oy$.  Thinking in terms of $\ox,\oy$, we have a factorization
\[
\Res_{t_3}(W_{1}^{L},W_{2}^{L}) = \content_{\{\ox,\oy\}}(\Res_{t_3}(W_{1}^{L},W_{2}^{L})) \cdot \Primpart_{\{\ox,\oy\}}(\Res_{t_3}(W_{1}^{L},W_{2}^{L})).
\]

The first factor is polynomial in $t_1,t_2$ only, while the second also depends on $\ox,\oy$.  Recall that when the curves intersect, the base points give intersection points that are independent of $\ox,\oy$.  Since the resultant takes multiplicities into account, this suggests that
\[
\betap = \deg_{\ot} (\content_{\{\ox,\oy\}}(\Res_{t_3}(W_{1}^{L},W_{2}^{L}))),
\]
which is proved in Theorem \ref{theorem-betaW1LW2L} for any $L \in \myG$.

To complete the proof of the degree formula \eqref{eq-degformula}, it remains to show that
\[
\deg(\myS)\cdot \mapdeg(\cP) = \deg_{\ot} (\Primpart_{\{\ox,\oy\}}(\Res_{t_3}(W_{1}^{L},W_{2}^{L}))),
\]
This is more challenging, since the line $H_1 \cap H_2$ has to be chosen carefully to meet the surface $\myS$ transversely.  In the body of the paper, we do this in two steps.  The first is the substitution
\[
(x_1,x_2,x_3,x_4,y_1,y_2,y_3,y_4) \to (x_4, 0, 0, -x_1, 0, 0, x_4, -x_3),
\]
which turns $W_1,W_2$ into $K_1,K_2$. The second step applies a carefully chosen $L \in \myG$ that does not affect the content and provides the needed transversality.  The result is
\[
\underbrace{\Res_{t_3}(K_{1}^{L},K_{2}^{L})}_{\begin{array}{l} \text{The degree}\\
\text{is } \deg(\cP)^2 \end{array}}
\ = \ \underbrace{\content_{\ox}(\Res_{t_3}(K_{1}^{L},K_{2}^{L}))}_{\begin{array}{l} \text{By Theorem \ref{theorem-BetaP-K1K2}, the}\\
\text{degree is } \betap \end{array}} \cdot  \underbrace{\Primpart_{\ox}(\Res_{t_3}(K_{1}^{L},K_{2}^{L}))}_{\begin{array}{l} \text{By Lemma \ref{lemma-formula}, the degree}\\
\text{is } \deg(\myS)\hskip1pt \mapdeg(\cP) \end{array}},
\]
where ``degree'' means ``degree in $\ot$''.
Notice how the careful choice of $L \in \myG$ described in Lemma \ref{lemma-formula}  involves the gradients needed to prove transversality.

It follows that the proof of \eqref{eq-degformula} given in Theorem~\ref{theorem-formulaP} is consistent with the argument from \cite{CoxWebPage} sketched in this appendix, though the proof of Theorem~\ref{theorem-formulaP} was discovered independently.

\section*{Acknowledgements}

%We would like to thank Prof. David Cox for his comments, suggestions and fruitful discussions. In special, we thank  David   for pointing out the existence of his second proof in \cite{CoxWebPage} and for explaining the differences between the proofs in \cite{Cox2001} and \cite{CoxWebPage}.

%\para
This work has been partially supported by FEDER/Ministerio de Ciencia, Innovación y Universidades-Agencia Estatal de Investigación/MTM2017-88796-P (Symbolic Computation:\
new challenges in Algebra and Geometry together with its applications).

\para The second and third authors belong to the Research Group ASYNACS (Ref.\ CT-CE2019/683).


\begin{thebibliography}{20}

\bibitem{AdkinsHoffmanWang2005} Adkins, W.A., Hoffman,  J. W.,  Wang, H.H.  (2005). {\it Equations of parametric surfaces with base points via syzygies}. J. Symbolic Comput.  Vol. 39. pp. 33--101.

\bibitem{BCD} Bus\'e, L., Cox, D.A., and D'Andrea, C. (2003). {\it Implicitization of surfaces in $\mathbb{P}^{3}$ in the presence of base points}. J. Algebra Appl. Vol. 2. pp. 189--214.

\bibitem{BH} Bruns, W. and Herzog, J. (1998). {\it Cohen-Macaulay Rings}. Cambridge University Press Cambridge.

\bibitem{CSV} Caravantes J., Sendra J.R., Sevilla D., Villarino C.  (2018).  {\it On the existence of birational surjective parametrizations of affine surfaces}. J. Algebra. Vol. 501. pp. 206-214.

\bibitem{Chen2005}   Chen, F., Cox, C.,  Liu, Y.  (2005). {\it The $\mu$-basis and implicitization of a rational parametric surface}. J. Symbolic Comput.  Vol. 39. pp. 689--706.

\bibitem{Conforto} Conforto, F. (1939). \textit{Le Superfici Razionali}. Zanichelli Bologna.

\bibitem{Cox2015}  Cox, D.A., Little, J.,  O'Shea, D.  (2015). {\it Ideals, Varieties and Algorithms}. Undergraduate Texts in Mathematics. Springer--Verlag New York.

\bibitem{Cox2001}  Cox, D.A.   (2001). {\it Equations of parametric curves and surfaces via syzygies}. Symbolic Computation: Solving Equations in Algebra, Geometry, and Engineering, Contemporary Mathematics. Vol. 286. AMS Providence RI. pp. 1--20

\bibitem{Cox2003} Cox, D.A. (2003). {\it Curves, surfaces and syzygies}.   Topics in Algebraic Geometry and Geometric Modeling, Contemporary Mathematics. Vol. 334.  AMS Providence RI.  pp. 131--150

\bibitem{CoxWebPage} Cox, D.A. (2001).  \textit{What is the Multiplicity of a Base Point?} Talk at the XIV Coloquio Latinoamericano de Algebra in the Summer of 2001 in Cordoba, Argentina.
  \url{https://dacox.people.amherst.edu/}

\bibitem{Fulton} Fulton, W. (1984). {\it Intersection Theory I}. Ergebnisse der Mathematik und ihrer Grenzgebiete. 3. Folge A Series of Modern Surveys in Mathematics. Springer-Verlag, Berlin Heidelberg.

 \bibitem{Harris:algebraic} Harris, J.   (1995). {\it Algebraic Geometry. A First Course}. Springer-Verlag New York.

\bibitem{Chen} Jia, X., Shi, X., Chen, F. (2018). {\it  Survey on the theory and applications of $\mu$-bases for rational curves and surfaces}. J. Comput. Appl. Math. Vol. 329. pp.  2--23.

%\bibitem{PDSeSi} P\'erez-D\'{\i}az, S., Sendra, J.R.,  Schicho, J. (2002). {\it Properness and Inversion of Rational Parametrizations of Surfaces}. Applicable Algebra in Engineering, Communication and Computing. Vol. 13. pp.29--51.

\bibitem{PS-grado}  P\'erez-D\'{\i}az, S.,  Sendra, J.R. (2004). {\it  Computation of the degree of rational surface parametrizations}. J. Pure Appl. Algebra. Vol. 193. pp.  99--121.

%\bibitem{PS-ISSAC} P\'erez-D\'{\i}az, S.,  Sendra, J.R. (2005).
%{\it Partial Degree Formulae for Rational Algebraic Surfaces}. Proc. ISSAC 2005. pp. 301--308. ACM Press.

\bibitem{JSC08} P\'erez-D\'{\i}az, S.,  Sendra, J.R. (2008).  {\it A Univariate resultant based implicitization algorithm for surfaces}. J. Symbolic Comput. Vol. 43. pp. 118--139.

\bibitem{PS13} P\'erez-D\'{\i}az, S.,  Sendra, J.R.  (2013). \textit{Behavior of the fiber and the base points of parametrizations under projections}.
Math. Comput. Sci. Vol. 7. pp. 167--184

%%\bibitem{PSV-Singularities} P\'erez--D\'{\i}az  S., Sendra J.R.,  Villarino C.   (2015).  {\it  Computing the Singularities of Rational Surfaces.}. Mathematics of Computation. Volume 84, Number 294, pp. 1991--2021.

\bibitem{SchichoGrado} Schicho, J. (1999). {\it A degree bound for the parameterization of a rational surface}. J. Pure Appl. Alg. Vol. 145. pp. 91--105.

\bibitem{SchichoIssac} Schicho, J. (2002). {\it Simplification of surface parametrizations}. ISSAC 2002: Proceedings of the 2002 International Symposium on Symbolic and Algebraic Computation. ACM Press New York.  pp. 229--237.

\bibitem{Schicho2006} Schicho, J. (2006). {\it The parametric degree of a rational surface} Math. Z. Vol. 254. pp. 185--198.

\bibitem{Sed1990} Sederberg, T.W.    (1990).  {\it Techniques for cubic algebraic surfaces I}.  Computer Graphics and Applications, IEEE 1.4 (199). pp. 14--25.

\bibitem{SSV14} Sendra, J.R., Sevilla, D., Villarino C. (2014). {\it Covering of surfaces parametrized without projective base points}. ISSAC 2014: Proceedings of the 2014 International Symposium on Symbolic and Algebraic computation. ACM Press New York. pp. 375--380.

\bibitem{SSV} Sendra, J.R., Sevilla, D., Villarino, C. (2017). {\it Covering rational ruled surfaces}. Math. Comp. Vol. 86, num. 308. pp. 2861--2875.

%\bibitem{Sen} Sendra, J.R.,   Winkler, F.  (2001).  {\it Computation of the Degree of a Rational Map}.  Proc. ISSAC-2001 (Ontario, Canada). pp. 317--322.
%%\bibitem{Sen2} Sendra J.R., Winkler F. (2001).  {\it Tracing Index of Rational Curve Parametrizations.}  Computer Aided Geometric Design Vol. 18/8,   pp. 771--795.

\bibitem{SWP}  Sendra, J.R., Winkler,  F., P\'erez-D\'{\i}az,  S.  (2007). {\it  Rational Algebraic Curves: A Computer Algebra Approach}. Algorithms and Computation in Mathematics.  Vol. 22. Springer Verlag New York.


\bibitem{Walker} Walker,  R.J.  (1950). {\it  Algebraic Curves.} Princeton Univ. Press Princeton NJ.

%\bibitem{Goldman-Quaternion} Wang, X., Goldman, R. (2015). {\it  Quaternion Rational Surfaces: Rational Surfaces Generated from the Quaternion Product of two Rational Space Curves.}  Graphical Models. Vol. 81, Issue C. pp. 18--32.

\bibitem{Wang2004} Wang, D.  (2004). {\it  A simple method for implicitizing rational curves and surfaces}. J. Symbolic Comput. Vol. 38. pp. 899--914.

\bibitem{Winkler} Winkler, F. (1996). {\it Polynomials Algorithms in Computer Algebra.}  Springer-Verlag Wien New York.

\bibitem{Sed2003} Zheng, J., Sederberg, T.W., Chionh, E.-W.,  Cox, D.A. (2003).  {\it Implicitizing rational surfaces with base points using the method of moving surfaces}.  Topics in Algebraic Geometry and Geometric Modeling, Contemporary Mathematics. Vol. 334. AMS Providence RI. pp. 151--168.


\end{thebibliography}
\end{document}